\documentclass[leqno,12pt]{amsart}

\usepackage{amsmath,amstext,amssymb,amsopn,amsthm,mathrsfs}

\usepackage{graphicx}

\title[Multi-dimensional pencil phenomenon]
    {The multi-dimensional pencil phenomenon for Laguerre heat-diffusion maximal operators $\!^{\ddag}$} 

\author[A{.} Nowak]{Adam Nowak}
\author[P{.} Sj\"ogren]{Peter Sj\"ogren}

\address{
\noindent Adam Nowak \newline
    Institute of Mathematics and Computer Science \newline
    Wroc\l{}aw University of Technology \newline
    Wyb{.} Wyspia\'nskiego 27
    PL--50--370 Wroc\l{}aw, Poland
    }
\email{\noindent anowak@pwr.wroc.pl}

\address{
\noindent Peter Sj\"ogren \newline
    Mathematical Sciences, 
    University of Gothenburg \newline
    Mathematical Sciences,
    Chalmers University of Technology \newline 
    SE-412 96 G\"oteborg, Sweden
    }
\email{\noindent peters@math.chalmers.se}

\allowdisplaybreaks
\theoremstyle{plain}
\newtheorem{thm}{Theorem}[section]

\newtheorem{lem}[thm]{Lemma}
\newtheorem{prop}[thm]{Proposition}

\theoremstyle{definition}

\theoremstyle{remark}
\newtheorem*{rem*}{Remark}

\theoremstyle{plain}

\def\R{\mathbb R}
\def\N{\mathbb N}
\def\Z{\mathbb Z}
\def\Rset{\mathbb R}

\def\M{\mathcal{H}_{*}^{\alpha}}

\begin{document}

\begin{abstract}
We investigate in detail the mapping properties of the maximal operator associated with 
the heat-diffusion semigroup corresponding to expansions with respect to multi-dimensional 
standard Laguerre functions $\mathcal{L}^{\alpha}_k$. Our interest is focused on the situation
when at least one coordinate of the type multi-index $\alpha$ is smaller than $0$. 
For such parameters $\alpha$ the Laguerre semigroup does not satisfy the general theory
of semigroups, and the behavior of the associated maximal operator on $L^p$ spaces is found
to depend strongly on both $\alpha$ and the dimension.
\end{abstract}

\maketitle

\footnotetext{ 
${\ddag}$ This paper was published as Preprint 2007:35, Department of Mathematical Sciences, 
    	Chalmers University of Technology and University of Gothenburg.

\emph{\noindent 2000 Mathematics Subject Classification:} primary 42C10; secondary 42B25\\
\emph{Key words and phrases:} maximal operators, Laguerre semigroups, 
	standard Laguerre expansions. 
	
	The first-named author was supported in part by MNiSW Grant N201 054 32/4285.
}

\section{Introduction}
\label{sec:intro}

Maximal operators play an important role in the theory of semigroups of operators.
In particular, their mapping properties are directly connected with the boundary
behavior of the semigroups. In this article we perform an extensive study of
the multi-dimensional maximal operator associated with a semigroup existing in 
the literature, but not covered by the general theory of Stein's monograph 
\cite{St}. The description we obtain is sharp and a bit unexpected. 
Certain endpoint results depend on the dimension of the underlying space; the situation 
changes drastically when one passes to dimension $4$ and higher.

The purpose of the paper is twofold. It provides results in the particular setting
of standard Laguerre expansions. But it also suggests methods and tools, and intuition,
for similar questions in other classical settings where the general theory does not apply.
This concerns for instance expansions in Laguerre functions of Hermite type and in 
ultraspherical or Jacobi ``orthonormalized''
polynomials, as well as certain Fourier-Bessel settings.    

Let $\{T_t^{\alpha}\}_{t > 0}$ be the heat-diffusion semigroup related to expansions with
respect to the $d$-dimensional standard Laguerre functions $\mathcal{L}_k^{\alpha}$
of type $\alpha \in (-1,\infty)^d$ (see Section \ref{sec:prel} for the relevant definitions).
The main object of our study is the family of maximal operators
$$
T^{\alpha}_{*}f = \sup_{t>0}|T^{\alpha}_t f|, \qquad \alpha \in (-1,\infty)^d.
$$
It is known that for $\alpha \in [0,\infty)^d$ the behavior of $T_*^{\alpha}$ is standard.
In fact we have the following result.
\begin{thm}[\cite{Stem},\cite{NoSj}] \label{StNoSj}
Let $d\ge 1$ and $\alpha \in [0,\infty)^d$. Then $T^{\alpha}_*$ is bounded on $L^p(\R^d_+)$
for $1<p\le \infty$ and of weak type $(1,1)$.
\end{thm}
In dimension one this was proved by Stempak \cite{Stem}, and
the multi-dimensional generalisation for $p>1$ is an easy consequence of Stempak's
result. The weak type $(1,1)$ in higher dimensions was obtained recently by 
the authors \cite{NoSj}. Note that $T^{\alpha}_*$ is not bounded on $L^1(\R^d_+)$.

However, when some $\alpha_i <0$, the $\mathcal{L}^{\alpha}_k$ are not in $L^p(\R^d_+)$
for all $1<p<\infty$. This suggests that $T^{\alpha}_t f$ is not
defined on $L^p(\R^d_+)$ for all $p$. 
To explain this phenomenon in greater detail, assume for clarity that $d=1$.
Let $G_t^{\alpha}(\xi,\eta)$ be the integral kernel of the semigroup $T_t^{\alpha}$
in dimension $1$.
This kernel can be expressed explicitly in terms of a modified Bessel function; see Section 
\ref{sec:prel}. In particular, $G_t^{\alpha}(\xi,\eta)$ behaves for each fixed $t>0$ like
$\xi^{\alpha\slash 2}\eta^{\alpha\slash 2}$ as $\xi,\eta \to 0^+$. The requirement that
$$
T_t^{\alpha}f(\xi) = \int G_t^{\alpha}(\xi,\eta) f(\eta)\, d\eta
$$
exists and is in $L^p(\R^d_+)$ for $f \in L^p(\R^d_+)$ implies, roughly speaking,
that $G^{\alpha}_t(\xi,\eta)$ is in $L^{p'}(d\eta)$ and in $L^p(d\xi)$, where
$1\slash p + 1\slash p' =1$. For $-1<\alpha<0$ this happens precisely when 
$|2\slash p -1| < \alpha+1$ or, equivalently,
$$
p_0(\alpha) < p < p_1(\alpha) \qquad \textrm{where} 
	\qquad p_1(\alpha) = - \frac{2}{\alpha} \qquad \textrm{and} \qquad
	p_0(\alpha) = p_1(\alpha)' = \frac{2}{2+\alpha}.
$$
Now two basic questions arise naturally: 
\begin{itemize}
\item[1)] Is $T^{\alpha}_*$ bounded on 
$L^p(\R_+)$ when $p_0 < p < p_1$? 
\item[2)]  Precisely what happens at the endpoints $p_0$ and $p_1$?
\end{itemize}
Both problems were studied recently by Mac\'{\i}as, Segovia and Torrea \cite{MST}, who proved
the following one-dimensional result (see also Figure \ref{fig:pencil} below).
\begin{thm}[\cite{MST}] \label{MaSeTo}
Let $d=1$ and $\alpha \in (-1,0)$. Then $T^{\alpha}_*$ is bounded on $L^p(\R_+)$ for
$p_0<p<p_1$, of weak type $(p_1,p_1)$ and of restricted weak type $(p_0,p_0)$.
Moreover, these results are sharp in the sense that $T^{\alpha}_*$ is neither bounded on
$L^{p_1}(\R_+)$ nor of weak type $(p_0,p_0)$ nor of restricted weak type $(p,p)$ for 
$p \notin [p_0,p_1]$.
\end{thm}
Recall that restricted weak type $(p_0,p_0)$ means that the inequality
$$
|\{x\in \R_+ : T^{\alpha}_*\chi_{E}(x)>\lambda\}| \le C \frac{|E|}{\lambda^{p_0}}, 
\qquad \lambda >0,
$$
holds (with a fixed $C$) for all sets $E \subset \R_+$ of finite measure. It is well known 
that this is equivalent to saying that $T_*^{\alpha}$ is bounded from $L^{p_0,1}(\R_+)$ to
the weak-type space $L^{p_0,\infty}(\R_+)$. Here the Lorentz space $L^{p_0,1}$ is equipped with the norm
$$
\|f\|_{p_0,1} = \int_0^{\infty} f^*(s) s^{1\slash p_0} \, \frac{ds}{s},
$$
and $f^*$ is the decreasing rearrangement of $f$ on $\R_+$.
The fact that this is indeed a norm is verified in \cite[Theorem 4.3 p.\,218]{BeSh}.

Thus, in one dimension, the study of the $L^p$ mapping properties of $T_*^{\alpha}$ is
complete for the full range $\alpha \in (-1,\infty)$. The results can be summarized
graphically, see Figure \ref{fig:pencil}.
\begin{figure} [ht]
\includegraphics[width=0.6\textwidth]{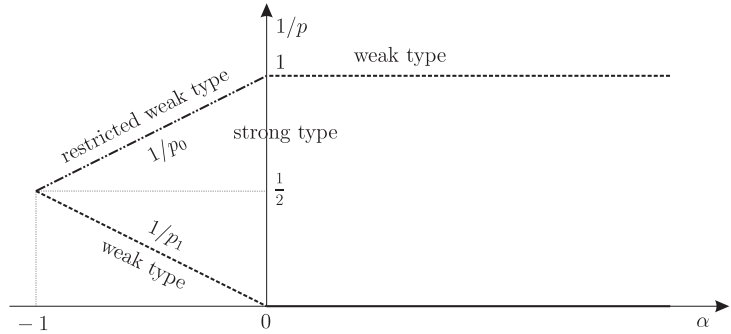}
\caption{The pencil phenomenon in one dimension.}\label{fig:pencil}
\end{figure}
In particular, the shape of a pencil appears, and this justifies the phrase
\emph{pencil phenomenon} sometimes used to describe the $L^p$ behavior of $T^{\alpha}_*$. 

The purpose of our present research is to study the mapping properties 
of $T_*^{\alpha}$ in arbitrary finite dimension $d$ and for any $\alpha \in (-1,\infty)^d$. 
The methods used in \cite{MST} are inadequate in higher dimensions.
The comprehensive, sharp result we establish, Theorem \ref{main_thm} below, turns out to be
rather intricate and unexpected.

Denote  
$$
\widetilde{\alpha} = \min_{1\le j \le d} \alpha_j, \qquad
\widetilde{d}(\alpha)=\#\big\{1\le i \le d : \alpha_i = \widetilde{\alpha}\big\},
\qquad p_1 = p_1(\widetilde{\alpha}), \quad p_0=p_0(\widetilde{\alpha}).
$$

\begin{thm} \label{main_thm}
Let $d \ge 1$, $\alpha \in (-1,\infty)^d$ and assume that $-1 < \widetilde{\alpha} < 0$.
\begin{itemize}
\item[(a)]
	If $\widetilde{d}(\alpha)=1$, then the results for $d=1$ remain valid for any $d$:
		\begin{itemize}
			\item[(a1)]
				If $p_0 < p < p_1$, then $T_*^{\alpha}$ is bounded on $L^p(\R^d_+)$.
			\item[(a2)]
				$T_*^{\alpha}$ is of weak type $(p_1,p_1)$.
			\item[(a3)]
				$T_*^{\alpha}$ is of restricted weak type $(p_0,p_0)$.
		\end{itemize}
\item[(b)]
	Assume now that $\widetilde{d}(\alpha) \ge 2$.
		\begin{itemize}
			\item[(b1)] 
				If $p_0 < p < p_1$, then $T_*^{\alpha}$ is bounded on $L^p(\R^d_+)$.
			\item[(b2)]
				For $d=2$ and $d=3$, $T^{\alpha}_*$ satisfies the logarithmic
				weak-type $(p_1,p_1)$ inequality
						$$
							\big|\{T_*^{\alpha}f > \lambda\}\big| \le C \frac{\|f\|_{p_1}^{p_1}}{\lambda^{p_1}}
								\bigg[ \log \bigg( 2 + \frac{\lambda}{\|f\|_{p_1}} 
									\bigg) \bigg]^{\widetilde{d}(\alpha)-1}, \qquad \lambda >0,
						$$ 
						for $f \in L^{p_1}(\R^d_+)$.
				But for $d \ge 4$, there exists an $f \in L^{p_1}(\R^d_+)$ such that
						$$
							\big|\{T_*^{\alpha}f > \lambda\}\big| = \infty
						$$
						for all $\lambda>0$. This $f$ can actually be taken in the
						smaller space $L^{p_1,1}(\R^d_+)$.
			\item[(b3)]
				For $d=2$ and $d=3$, $T_*^{\alpha}$ satisfies the logarithmic restricted 
				weak-type $(p_0,p_0)$ estimate
						$$
							\big|\{T^{\alpha}_* \chi_E > \lambda \}\big| \le C \frac{|E|}{\lambda^{p_0}}
								\bigg[ \log\bigg(2+ \frac{1}{|E|}\bigg) 
									\bigg]^{\frac{p_0}{p_1}(\widetilde{d}(\alpha)-1)}, \qquad \lambda > 0,
						$$
					for all measurable sets $E \subset \R^d_+$ of finite measure.
				But for $d \ge 4$
				this estimate does not hold, even if the exponent of the logarithmic factor is
				arbitrarily increased.
		\end{itemize}	
\end{itemize}
\end{thm}

The boundedness properties in (b2) and (b3) can be stated in a compact way in
terms of function spaces. More precisely, the estimate in (b2) is equivalent to
the boundedness of $T_*^{\alpha}$ from $L^{p_1}(\R^d_+)$ to the weak-type Orlicz space
$L^{p_1,\infty}\log^{-N\slash p_1} L$, with $N=\widetilde{d}(\alpha)-1$, 
defined to consist of those measurable functions $f$ for which the quasinorm
$$
\|f\|_{L^{p_1,\infty}\log^{-N\slash p_1} L} = \inf\Big\{ \eta > 0 :
\sup_{\lambda>0}  \lambda \log^{-N\slash p_1}(2+\lambda) \,
\big|\{x \in \R^d_+ : |f(x)| \slash \eta > \lambda\}\big|^{1\slash p_1} \le 1 \Big\}
$$
is finite. The estimate in (b3) is equivalent to the boundedness of $T^{\alpha}_*$ from
$L^{p_0,1}\log^{N\slash p_1} L$ to $L^{p_0,\infty}(\R^d_+)$, where the Lorentz-Zygmund
space $L^{p_0,1}\log^{N\slash p_1} L$ is defined by means of the quasinorm
$$
\|f\|_{L^{p_0,1}\log^{N\slash p_1} L} =
\int_0^{\infty} f^*(s) s^{1\slash p_0} \log^{N\slash p_1}\Big(2+\frac{1}{s}\Big)\,
	 \frac{ds}{s}.
$$
When $\widetilde{d}(\alpha)=1$, the logarithmic factors disappear and we are reduced to
the classic Lorentz spaces which appear implicitly in (a2) and (a3). 
Thus parts (a) and (b) are consistent.

The sharpness of items (b2) and (b3) for $d=2,3$ is discussed in
Sections \ref{sharp_p1} and \ref{sharp_P0} below. The Orlicz and Lorentz-Zygmund
spaces just defined are found to be best possible here, for large $\lambda$ and small
$s$, respectively. Also (a2) and (a3) are sharp in a similar way. 
No boundedness holds for $p \notin [p_0,p_1]$.

Summarizing the picture, we see that the behavior of $T_*^{\alpha}$ depends
in an essential way on $\alpha$, via the quantities $\widetilde{\alpha}$ and
$\widetilde{d}(\alpha)$. When $\widetilde{\alpha} \ge 0$, the results are standard
(Theorem \ref{StNoSj}). In the opposite case, everything depends on $\widetilde{d}(\alpha)$,
the number of minimal values $\alpha_i$  in the type multi-index 
$\alpha = (\alpha_1,\ldots,\alpha_d)$. 
If there is only one minimal value (notice that this always happens when $d=1$), 
the results are analogous to those obtained earlier in dimension one; see Theorem \ref{MaSeTo}.
Otherwise, the dimension $d$ of the underlying space comes into play.
For $d=2,3$ there are sharp boundary results expressed by means of appropriate function spaces.
But for dimension $4$ and higher, there are no similar endpoint results.
Roughly speaking, it turns out that for $d \ge 4$, but not for smaller $d$, 
there is enough room in the space for counterexamples.
The joint effect of at least two minimal values $\alpha_i$ is then needed.

Finally, we point out that our proof of Theorem \ref{main_thm} contains an argument proving
Theorem \ref{MaSeTo}, shorter than the original proof in \cite{MST}. On the other hand,
we mention that recently Mac\'{\i}as, Segovia and Torrea \cite{MST2}, still working in 
dimension one, obtained sharp power-weighted results in the spirit of Theorem \ref{MaSeTo},
by extending the methods used in \cite{MST}.

This paper is organized as follows. Section \ref{sec:prel} describes the setup and gathers 
basic lemmas providing, in particular, fundamental kernel estimates.
The remaining part of the paper constitutes the proof of Theorem \ref{main_thm}.
Thus Section \ref{sec:strong} treats the strong-type range $p_0<p<p_1$
(items (a1) and (b1) of Theorem \ref{main_thm}), Section \ref{sec:p1} deals with
the endpoint $p_1$ (items (a2) and (b2)) and Section \ref{sec:p0} with the endpoint
$p_0$ (items (a3) and (b3)). Comments on the sharpness of items (a2), (a3), (b2), (b3)
are located in the final parts of Sections \ref{sec:p1} and \ref{sec:p0}.

We shall use the following conventions. By $c>0$ and $C< \infty$ we will always denote
constants whose values may change from one occurrence to another; these constants will
usually depend on the dimension $d$ and the type multi-index $\alpha.$ Any other
dependence will usually be indicated. If $c \le f\slash g \le C$ for some $c$ and $C,$ we will
write shortly $f \simeq g.$ Similarly, we abbreviate $f \le C g$ to $f \lesssim g$ and
write $\gtrsim$ analogously. Given a $d$-tuple $m = (m_1,\ldots,m_d) \in \Rset^d$, 
its length $m_1+\ldots + m_d$ will be denoted by $|m|$ (notice that this quantity may be
negative). For $p \in [1,\infty]$ the norm in the Lebesgue space $L^p=L^p(\R^d_+)$ is denoted
by $\|\cdot\|_p$. We write $p'$ for the conjugate exponent of $p$.

\section{Preliminaries}
\label{sec:prel}

Let $L_k^{\alpha}$ denote the one-dimensional Laguerre polynomials with parameter
$\alpha>-1$, defined on $\R_+ = (0,\infty)$ for $k=0,1,2,\ldots$; cf. \cite[Chapter 4]{Leb}.
We consider the system $\{\mathcal{L}_k^{\alpha}\}$ of \emph{standard Laguerre functions},
given in dimension one by
$$
\mathcal{L}_k^{\alpha}(x) = \bigg(\frac{k!}{\Gamma(k+\alpha+1)}\bigg)^{1\slash 2} L^{\alpha}_k(x)
	 x^{\alpha\slash 2}e^{-x\slash 2}, \qquad x>0, \qquad \alpha >-1, \qquad k \in \N.
$$
The corresponding multi-dimensional systems are formed by taking tensor products.
It is well known that
$\{\mathcal{L}_k^{\alpha} : k \in \N^d\}$ is an orthonormal basis in $L^2(\R^d_+)$,
for any $\alpha \in (-1,\infty)^d$.

Each $\mathcal{L}^{\alpha}_k$ is an eigenfunction of the differential operator
$$
\mathbb{L}^{\alpha} = - \sum_{i=1}^d \bigg(x_i \frac{\partial^2}{\partial x_i^2} +
    \frac{\partial}{\partial x_i} - \frac{x_i}{4} - \frac{\alpha_i^2}{4x_i} \bigg),
$$
which is formally symmetric and positive, and the corresponding eigenvalue is
$|k|+(|\alpha| + d)\slash 2$. Moreover, $\mathbb{L}^{\alpha}$ has a self-adjoint extension
in $L^2(\R^d_+)$ for which the spectral decomposition is given by the $\mathcal{L}^{\alpha}_k$.
Hence, the associated heat semigroup $T_t^{\alpha} = \exp(-t\mathbb{L}^{\alpha})$ 
is defined for $f \in L^2(\R^d_+)$ by
$$
	T^{\alpha}_t f(x) = \sum_{n=0}^{\infty} 
		\exp\Big(-t\frac{2n+|\alpha|+d}{2}\Big) \sum_{|k|=n}
    \langle f, \mathcal{L}^{\alpha}_k \rangle \mathcal{L}^{\alpha}_k(x),
    \qquad t>0, \quad x \in \Rset^d_{+}.
$$
It follows from the Hille-Hardy formula \cite[(4.17.6)]{Leb} that the integral representation is
\begin{equation} \label{hi}
	T^{\alpha}_t f(x) = \exp\Big({-t\frac{|\alpha|+d}{2}}\Big)
	\int {\mathcal{H}}^{\alpha}_{t\slash 2}(x,y)f(y) \, dy,
\end{equation}
where the kernel is given by
${\mathcal{H}}^{\alpha}_t(x,y) =  \prod_{i=1}^d {\mathcal{H}}^{\alpha_i}_{t}(x_i,y_i)$,
with the component kernels
$$
\mathcal{H}_t^{a}(\xi,\eta) = \frac{\exp({(1+a)t})}{2\sinh t}
	\exp\Big({-\frac{\coth t}{2} (\xi+\eta)}\Big) I_{a}\Big(\frac{\sqrt{\xi \eta}}{\sinh t}\Big)
$$
for $\xi, \eta >0$, $t>0$ and $a>-1$. Here $I_a$ is the modified Bessel function of the first
kind and order $a$. Note that $\mathcal{H}_t^{\alpha}(x,y)$
is strictly positive for $x,y \in \R^d_{+}$, $t>0$.

We now deduce some useful upper and lower estimates for the one-dimensional kernel
$\mathcal{H}_t^a(\xi,\eta)$. 
By using the standard asymptotics (cf. \cite[(5.16.4), (5.16.5)]{Leb})
\begin{equation} \label{asymp}
I_{a}(x) \simeq x^{a}, \quad x \to 0^{+}, \qquad
I_{a}(x) \simeq x^{-1\slash 2}e^x, \quad x \to \infty,
\end{equation}
and the fact that $I_{a}(\cdot)$ is continuous on $(0,\infty),$ we see that
$$
\mathcal{H}^{a}_t(\xi,\eta) \simeq 
\left\{ \begin{array}{ll} D^{a}_t(\xi,\eta) &  \textrm{if} \;\; \sqrt{\xi \eta} \le \sinh t \\
													E^a_t(\xi,\eta) & \textrm{if} \;\; \sqrt{\xi \eta} > \sinh t 
				\end{array}
\right.,
\qquad \xi,\eta>0, \quad t>0, \quad a > -1,
$$
where
\begin{align*}
	D_t^a(\xi,\eta) & = \frac{\exp((1+a)t)}{(\sinh t)^{a+1}} (\xi \eta)^{a\slash 2} 
		\exp\Big({-\frac{\coth t}{2} (\xi+\eta)}\Big),\\
	E_t^a(\xi,\eta) & = \frac{\exp((1+a)t)}{(\sinh t)^{1\slash 2}} (\xi \eta)^{-1\slash 4} 
		\exp\Big({-\frac{\coth t}{2} (\xi+\eta)} + \frac{\sqrt{\xi \eta}}{\sinh t} \Big) \\
			& = \frac{\exp((1+a)t)}{(\sinh t)^{1\slash 2}} (\xi \eta)^{-1\slash 4} 
		\exp\Big(\frac{-1}{2\sinh t} \frac{(\xi-\eta)^2}{(\sqrt{\xi}+\sqrt{\eta})^2}\Big) 
			\exp\Big( \frac{1-\cosh t}{2\sinh t} (\xi + \eta) \Big).
\end{align*}

\begin{lem} \label{upper}
For all $0 < t \leq 1$ and $\xi, \eta > 0$ we have
\begin{align} \label{main}
	\mathcal{H}_t^a(\xi,\eta) \lesssim \frac1{\sqrt{t\xi}} 
	\exp\Big(-c\,\frac{(\xi-\eta)^2}{t\xi}\Big) \exp\big(-ct(\xi+\eta)\big)
	+ \frac{(\xi \eta)^{a/2}}{t^{a+1}} \exp\Big(-c\,\frac{\xi+\eta}{t}\Big).
\end{align}
Similarly, for $t > 1$ and $\xi, \eta > 0$,
\begin{align} \label{large_t}
	\mathcal{H}_t^a(\xi,\eta) \lesssim (\xi \eta)^{a/2} \exp\big(-c(\xi+\eta)\big).
\end{align}
\end{lem}

Observe that the right-hand side in \eqref{large_t} coincides with the last term in 
\eqref{main} taken with $t=1$. So for the maximal operator defined in \eqref{max_H} below
by taking the supremum over $t > 0$, it is enough to consider the right-hand side of
(\ref{main}) with  $0 < t \leq 1$.

\begin{proof}[Proof of Lemma \ref{upper}]
We consider the two cases obtained from the asymptotics of $I_a$. \\
\noindent \textit{Case 1}: $\sqrt{\xi \eta} \le \sinh t$. 
Then $\mathcal{H}_t^a(\xi,\eta) \simeq D_t^a(\xi,\eta)$ and it is immediately
seen that for $t\le 1$
$$
\mathcal{H}_t^a(\xi,\eta) \lesssim \frac{1}{t^{a+1}} (\xi \eta)^{a\slash 2} 
	\exp\Big( -c\frac{\xi + \eta}{t}\Big).
$$ 
On the other hand, when $t>1$ we easily get
$$
\mathcal{H}_t^a(\xi,\eta) \lesssim (\xi \eta)^{a\slash 2} \exp\big( -c(\xi+\eta) \big).
$$
\noindent \textit{Case 2}: $\sqrt{\xi \eta} > \sinh t$. 
Now $\mathcal{H}_t^a(\xi,\eta) \simeq E_t^a(\xi,\eta)$. Assume first that $t \le 1$. 
Then we have for $\xi \simeq \eta$
$$
\mathcal{H}_t^a(\xi,\eta) \lesssim \frac{1}{\sqrt{t\xi}}
	\exp\Big(-c\,\frac{(\xi-\eta)^2}{t\xi}\Big) \exp\big(-ct(\xi+\eta)\big).
$$
If $\xi \simeq \eta$ does not hold, then 
$(\xi-\eta)^2\slash (\sqrt{\xi}+\sqrt{\eta})^2 \simeq \xi + \eta$ and we can write
\begin{align*}
\mathcal{H}_t^a(\xi,\eta) & \lesssim \frac{1}{t^{1\slash 2}(\xi \eta)^{1\slash 4}} 
	\exp\Big( -c\,\frac{\xi+\eta}{t}\Big) \\
& < \Big( \frac{\sqrt{\xi\eta}}{\sinh t} \Big)^{a+1}
	\frac{1}{t^{1\slash 2}(\xi \eta)^{1\slash 4}} \exp\Big( -c\,\frac{\xi+\eta}{t}\Big)\\
& \simeq \frac{1}{t^{a+1}} (\xi \eta)^{a\slash 2} \exp\Big( -c\,\frac{\xi+\eta}{t}\Big)
	\Big(\frac{\sqrt{\xi \eta}}{t}\Big)^{1\slash 2} \\
& \lesssim \frac{1}{t^{a+1}} (\xi \eta)^{a\slash 2} \exp\Big( -{c}\,\frac{\xi+\eta}{t}\Big),
\end{align*}
with a new value of $c$ in the last expression. Letting $t>1$, we get
\begin{align*}
\mathcal{H}_t^a(\xi,\eta) & \lesssim \frac{\exp((1+a)t)}{(\sinh t)^{1\slash 2} 
(\xi \eta)^{1\slash 4}}
	\exp\big( -c(\xi+\eta) \big) \\
& \simeq (\xi \eta)^{a\slash 2} \exp\big( -c(\xi+\eta) \big) 
	\Big( \frac{\sinh t}{\sqrt{\xi \eta}}\Big)^{a+1} 
		\Big(\frac{\sqrt{\xi \eta}}{\sinh t}\Big)^{1\slash 2} \\
& \lesssim (\xi \eta)^{a\slash 2} \exp\big( -{c}(\xi+\eta) \big),
\end{align*}
since $\sinh t \slash \sqrt{\xi \eta}<1$, $\sqrt{\xi \eta} < \xi + \eta$ and $\sinh t >1$. 
Altogether, this proves the lemma.
\end{proof}
  
The next result shows that Lemma \ref{upper} is sharp in certain cases.

\begin{lem} \label{lower}
The following lower estimates hold.
\begin{itemize}
\item[(a)]  
For $0 < t \leq 1\slash 4$ and $(2t)^{-1} < \xi < 2t^{-1}$ and $|\xi-\eta| < 1$,
$$
\mathcal{H}_t^a(\xi,\eta) \gtrsim 1.
$$
\item[(b)] 
For $0 < t \leq 1$ and $0 < \xi < 2t$ and $0 < \eta < 2t$,
$$
\mathcal{H}_t^a(\xi,\eta) \gtrsim \frac{(\xi \eta)^{a\slash 2}}{t^{a+1}}.
$$
\end{itemize}
\end{lem}

\begin{proof}
Under the assumptions of (a), we also have $\xi \simeq \eta$ and 
$\sqrt{\xi \eta}> t \simeq \sinh t$, hence
$$
\mathcal{H}_t^a(\xi,\eta) \simeq E_t^a(\xi,\eta) \simeq
		\frac{1}{t^{1\slash 2}(\xi \eta)^{1\slash 4}} 
		\exp\Big(\frac{-1}{2\sinh t} \frac{(\xi-\eta)^2}{(\sqrt{\xi}+\sqrt{\eta})^2}\Big) 
			\exp\Big( \frac{1-\cosh t}{2\sinh t} (\xi + \eta) \Big) \gtrsim 1.
$$
Considering (b), we have $\xi + \eta \lesssim t \simeq 1\slash \coth t$ and therefore
$$
\mathcal{H}_t^a(\xi,\eta) \simeq D_t^a(\xi,\eta) \simeq \frac{(\xi\eta)^{a\slash 2}}{t^{a+1}}
	\exp\Big( - \frac{\coth t}{2} (\xi+\eta) \Big) \gtrsim 
	\frac{(\xi\eta)^{a\slash 2}}{t^{a + 1}}.
$$
\end{proof}

The following technical result provides estimates of certain level sets 
that will be crucial in further developments.

\begin{lem} \label{level}
\begin{itemize}
\item[(a)] For $t,\nu>0$, one has
$$
\bigg|\bigg\{x \in (0,t)^d: \prod_{j=1}^d x^{-1}_j > \nu  \bigg\}\bigg| \le 
	\frac{C}{\nu} \big[\log\big(2 + t^{d}\nu\big)\big]^{d-1},
$$
where $C$ depends only on the dimension.
\item[(b)] 
For $t^{d} \nu \ge 1$, the estimate in {\rm (a)} is sharp in the sense 
that the level set has measure at least $c \nu^{-1} [\log(2 + t^{d} \nu)]^{d-1}$,
with $c=c(d)>0$.
\item[(c)] 
Given $\gamma > 0$, one has for $\sigma,\nu>0$ 
$$
\bigg|\bigg\{x \in \mathbb{R}_+^d: \prod_{j=1}^d x^{-1}_j \; 
\exp\bigg(- \Big(\sigma \sum_{j=1}^d x_j\Big)^\gamma\bigg) > \nu  \bigg\}\bigg| 
	\le \frac{C}{\nu} \big[\log\big(2 + \sigma^{-d} \nu\big)\big]^{d-1},
$$
with $C=C(\gamma,d)$.
\end{itemize}
\end{lem}

\begin{proof}
In (a) and (b) one may assume that $t = 1$, since the general case then
follows by a simple scaling argument. To prove (a), we shall use induction in $d$,
observing that the case $d = 1$ is obvious. The induction assumption implies that the set 
$$
\bigg\{x \in (0,1)^d: \prod_{j=1}^{d-1} x^{-1}_j > \nu  \bigg\}
$$
has measure at most $C \nu^{-1} [\log(2 + \nu)]^{d-2}$, so that this set 
can be neglected. The remaining set to be considered is then 
$$
\bigg\{ x \in (0, 1)^d:  x_d < \nu^{-1}\prod_{j=1}^{d-1} x^{-1}_j < 1\bigg\},
$$
where $\nu \ge 1$; otherwise the set is empty.
Clearly, its measure is
\[
\int\!\! \cdots \!\!\int \nu^{-1}{\prod_{j=1}^{d-1} x^{-1}_j} \, dx_1 \cdots dx_{d-1},
\]
the integral taken over the set $(0, 1)^{d-1} \cap \{\prod_{j=1}^{d-1} x^{-1}_j < \nu\}$. 
This set is contained in the square $(\nu^{-1}, 1)^{d-1}$, and so the integral is not
larger than $\nu^{-1} (\log \nu)^{d-1}$. This proves (a). 
Moreover, the same set contains the square $(\nu^{-1/(d-1)}, 1)^{d-1}$ and
thus the integral is not less than $c \nu^{-1}(\log \nu)^{d-1}.$ Now (b) follows.

In (c) we can assume that $\sigma = 1$, because of a scaling argument.
Then for each $k \in \{0, 1,\ldots\}$, the intersection of the level set in
(c) with the band $k < \sum_{j=1}^d x_j < k+1$ is contained in the set
$$
\bigg\{x \in (0,k+1)^{d}: \prod_{j=1}^d x^{-1}_j > \nu \exp\big(k^\gamma\big)\bigg\}.
$$
In view of (a), the measure of this set is no larger than
$$
C \nu^{-1} \exp(- k^\gamma) \log \big(2 + (k+1)^d\nu \exp(k^\gamma)\big)^{d-1}.
$$
Summing these quantities in $k$, we get at most $C\nu^{-1} [\log(2 + \nu)]^{d-1}$, and (c) follows.
\end{proof}

From now on, we shall work with the maximal operator
\begin{equation} \label{max_H}
\M f(x) = \sup_{t>0} \int \mathcal{H}_t^{\alpha}(x,y) |f(y)|\,dy
\end{equation}
rather than with $T^{\alpha}_*$. Since $\M f$ dominates $T_*^{\alpha} f$,
we shall obtain even slightly stronger positive results than stated in Theorem 
\ref{main_thm}. Also counterexamples will be constructed for $\M$. In those only small
values of $t$ will be used, and since the two maximal operators are comparable if $t$ is
restricted to, say, $(0,1)$, the counterexamples will be valid for $T^{\alpha}_*$ as well.

Let $M_j$ denote the standard centered one-dimensional maximal operator in 
$\mathbb{R}_+^d$, taken with respect to the $j$th variable.

\section{The range $p_0 < p < p_1$} \label{sec:strong}

We first consider the one-dimensional case, assuming that $\alpha = a \in (-1,0)$.
The critical exponents are $p_1 = -2\slash a$ and $p_0 = p_1' = 2\slash (a+2)$.

\begin{prop} \label{estimate} 
Let $d=1$ and $-1<a<0$. There exists a constant $c$ such that, for $0 < t \leq 1$ 
and any suitable function $f$ defined in $\mathbb{R}_+$,
\begin{align} \label{est}
\int \mathcal{H}_t^a(\xi,\eta) |f(\eta)|\,d\eta & \lesssim  
	e^{-ct\xi} M_1f(\xi) + e^{-c\xi/t} \xi^{-1/p_1} \|f\|_{p_1},\\
\nonumber
\int \mathcal{H}_t^a(\xi,\eta) |f(\eta)|\,d\eta & \lesssim  
		e^{-ct\xi} M_1f(\xi) + e^{-c\xi\slash t}\xi^{-1/p_0} \|f\|_{p_0,1}.
\end{align}
For $t > 1$, the same inequalities hold with $t$ replaced by 1 in the right-hand sides.
\end{prop}

Suppressing the exponentials, we immediately deduce from these estimates 
the weak type $(p_1, p_1)$ and the restricted weak type $(p_0, p_0)$ of the one-dimensional
maximal operator $\M$. Then, by interpolation, it follows that $\M$ is bounded on $L^p(\R_+)$
for $p_0<p<p_1$. This implies the known results for 
$T^{\alpha}_*$ stated in Theorem \ref{MaSeTo}.

\begin{proof}[Proof of Proposition \ref{estimate}]
We integrate \eqref{main} against $|f(\eta)|$. The first term of the right-hand side in
\eqref{main} leads to an integral which can be estimated by $C e^{-ct\xi} M_1f(\xi)$.
The second term of \eqref{main} gives an integral which can be dominated by
\begin{align*}
e^{-c\xi\slash t} \frac{\xi^{a/2}}{t^{a+1}} \int_0^\infty \eta^{a/2} 
	e^{-c {\eta}\slash t}|f(\eta)|\,d\eta
& \le  e^{-c\xi\slash t} \frac{\xi^{-1\slash p_1}}{t^{a+1}} \bigg(\int_0^\infty 
	\eta^{a p_0\slash 2} e^{-c p_0 {\eta}\slash t}\,d\eta \bigg)^{1\slash p_0} \|f\|_{p_1}\\
& \simeq  e^{-c\xi\slash t} \xi^{-1/p_1} \|f\|_{p_1},
\end{align*}
since $a\slash 2 + 1\slash p_0 = a+1$. But the same integral is also dominated by
\begin{align*}
e^{-c\xi\slash t} \frac{\xi^{a/2}}{t^{a+1}} \int_0^\infty \eta^{a/2} 
	|f(\eta)|\,d\eta
& =  e^{-c\xi\slash t} \Big(\frac{\xi}{t}\Big)^{a+1} \xi^{-1\slash p_0} 
	\int_0^\infty \eta^{-1\slash p_1} |f(\eta)| \,d\eta \\
& \lesssim  e^{-{c}\xi\slash t} \xi^{-1/p_0} 
	\int_0^\infty |f(\eta)| \eta^{1\slash p_0} \,\frac{d\eta}{\eta}.
\end{align*}
Since $\int f g \le \int f^* g^*$ (this is a well-known inequality due to Hardy and Littlewood,
see \cite[Theorem 2.2]{BeSh}), the last integral is controlled by $\|f\|_{p_0,1}$.
The conclusion follows in the case $t \le 1$, and for $t>1$ it is a consequence of \eqref{large_t}.
\end{proof}

Let now $d\ge 1$ and $-1 < \widetilde{\alpha} < 0$.
The critical endpoints are $p_1=p_1(\widetilde{\alpha})$ and $p_0=p_0(\widetilde{\alpha})$.
The next result justifies items (a1) and (b1) in Theorem \ref{main_thm}.
\begin{thm} \label{t_strong}
Let $d\ge 1$, $\alpha \in (-1,\infty)^d$ and assume that $-1 < \widetilde{\alpha} < 0$.
Then $\M$ is bounded on $L^p(\R^d_+)$ for $p_0<p<p_1$.
\end{thm}

\begin{proof}
We shall use the tensor product structure of $\mathcal{H}_t^{\alpha}(x,y)$
and the one-dimensional results. Observe that by Fubini's theorem
$$
\M f(x) \le \big( \mathcal{H}_*^{\alpha_1} \circ \cdots \circ \mathcal{H}_*^{\alpha_d} \big) f(x),
	\qquad x \in \R_+^d,
$$
where $\mathcal{H}_*^{\alpha_i}$ is the one-dimensional maximal operator acting on the $i$th
coordinate. Moreover, for each $i=1,\ldots,d$, we have 
$p_0(\alpha_i) \le p_0 < p_1 \le p_1(\alpha_i)$, where for $\alpha_i \ge 0$ we let 
$p_0(\alpha_i)=1$ and $p_1(\alpha_i)=\infty$. Thus it suffices to show that each
$\mathcal{H}_*^{\alpha_i}$ is bounded on $L^p$ provided that $p_0(\alpha_i)<p<p_1(\alpha_i)$.

In the case when $-1<\alpha_i<0$ this follows from Proposition \ref{estimate}, as commented
above. When $\alpha_i \ge 0$ it is enough to justify boundedness of $\mathcal{H}_*^{\alpha_i}$
on $L^{\infty}$ and from $L^1$ to $L^{1,\infty}$, since then the $L^p$ boundedness 
will follow by interpolation. The relevant $L^{\infty}$ result, however, is readily derived
from equation (3.7) in \cite{Stem} (we remark that $C$ is missing there), which implies
$$
\mathcal{H}_*^{\alpha_i} \boldsymbol{1}(x_i) = 
	\sup_{t>0} \int \mathcal{H}_t^{\alpha_i}(x_i,y_i)\, dy_i \le C, \qquad x_i \in \R_+.
$$
On the other hand, the weak type $(1,1)$ of $\mathcal{H}_*^{\alpha_i}$ was proved recently
by the authors, see \cite[Section 3.2]{NoSj}.
\end{proof}

Note that if $\widetilde{\alpha}\ge 0$ then $\M$ is bounded on $L^p(\R_+^d)$, $1<p\le \infty$,
and from $L^1(\R_+^d)$ to $L^{1,\infty}(\R_+^d)$, which is slightly stronger than
the statement of Theorem \ref{StNoSj}, see \eqref{hi}. This follows from the above estimate
of $\mathcal{H}_*^{\alpha_i} \boldsymbol{1}$, the weak type $(1,1)$ results in 
\cite[Section 3.2]{NoSj} and interpolation.

\section{The endpoint $p_1$} \label{sec:p1}

We work in dimension $d$ and assume that $-1<\widetilde{\alpha}<0$. 
The maximal operator under consideration is $\M$. We first observe that the results already 
obtained in Section \ref{sec:strong} can be used, in a straightforward
manner, to deal with the situation when there is only one minimal value $\alpha_i$. 
Indeed, assume that $\alpha_1$ is the only minimal $\alpha_i$. Due to the product structure
of $\mathcal{H}_t^{\alpha}(x,y)$, it suffices to use first the strong type $(p_1,p_1)$
estimate in the variables $x_2,\ldots, x_d$, and then apply the one-dimensional weak-type
result in $x_1$. This gives item (a2) of Theorem \ref{main_thm}.

Proving the remaining positive results is more complicated. For the sake of clarity,
we consider two main cases: when all $\alpha_i$ are equal (and so minimal) and when there
are precisely two minimal $\alpha_i$ in dimension $d=3$. This will be enough to prove the
positive part of Theorem \ref{main_thm} (b2). Counterexamples
in dimension $d=4$ and higher will be given at the end of this section.

\subsection{The case when all $\alpha_i$ are minimal} \label{ssec:amp1} \qquad \\
We assume that $d \ge 2$ and let $\widetilde{\alpha}=a$, with $-1<a<0$. The critical
exponents are $p_1 = -2/a$ and $p_0 = p'_1 = 2/(a+2)$.

\begin{thm} \label{pos}
Assume that $\alpha_i = a \in (-1,0)$ for all $i$. Then for $d = 2, 3$, the operator $\M$ maps 
$L^{p_1}$ boundedly into the space weak $L^{p_1}\log^{-(d-1)\slash p_1} L$, in the sense 
that for $f \in L^{p_1}$ the distribution function of $\M f$ satisfies
\begin{equation} \label{eq:positive}
	\big|\{\M f > \lambda\}\big| \le C \frac{\|f\|_{p_1}^{p_1}}{\lambda^{p_1}} 
	\bigg[\log\bigg(2 + \frac{\lambda}{\|f\|_{p_1}}\bigg)\bigg]^{d-1}, \qquad \lambda > 0.
\end{equation}
\end{thm}

To estimate $\int \mathcal{H}_t^{\alpha}(x,y) |f(y)|\,dy$,
we shall integrate one variable $y_j$ at a time, and apply Proposition \ref{estimate}
in each variable. This will produce a sum of $2^d$ terms, since for each variable we consider 
separately the two terms in the right-hand side of \eqref{est}.

To describe these $2^d$ terms, it is convenient to introduce the following notation. 
Let $D'$ be a subset of $\{1,\ldots, d\}$, and write $D''$ for its complement. 
By $d'$ and $d'' = d-d'$ we denote the number of elements of $D'$ and $D''$, respectively.
Given $x \in \R_+^d$, we let $x' \in \R_+^{d'}$ consist of those coordinates $x_j$ with 
$j\in D'$, and similarly for $x'' \in \R_+^{d''}$. Thus we can write $x = (x', x'')$.
We also denote
$$
M'' = \prod_{j\in D''}M_j,
$$
and observe that this product of one-dimensional maximal operators is bounded on 
$L^p(\R_+^{d''})$ for $1<p<\infty$.

When integrating $\mathcal{H}_t^{\alpha}(x,y) |f(y)|$ with respect to $y_j$, we consider the second term 
in the right-hand side of \eqref{est} if $j\in D'$, and the first term if $j\in D''$. 
Integrating in $y''$ before $y'$, we are led to expressions
\begin{equation} \label{eq:term}
T_t^{D'} f(x) = \exp\bigg(-\frac{c}{t}\sum_{j\in D'}x_j - ct\sum_{j\in D''}x_j\bigg) 
\prod_{j\in D'} x_j^{-1/p_1}\, \|M'' f(\cdot,x'')\|_{L^{p_1}(\R_+^{d'})}.
\end{equation}
Altogether we conclude
$$
\int \mathcal{H}_t^{\alpha}(x,y) |f(y)|\,dy \lesssim \sum_{D'} T_t^{D'} f(x);
$$
here the sum is taken over all possible choices of $D'$. 
We shall find estimates for each operator 
$$
T_{*}^{D'} f(x) = \sup_{0<t\leq 1} T_t^{D'} f(x); 
$$
clearly, $\M f \lesssim \sum_{D'} T_*^{D'}f$. Observe first that in the simple case when 
$D' = \emptyset$, the operator $T_{*}^{D'}$ is bounded on $L^{p_1}$. At the opposite 
extreme when  $D' = \{1,\ldots, d\}$, we can estimate the exponential in \eqref{eq:term}
by $\exp(-c \sum x_j)$. Then from  Lemma \ref{level} (c) with $\gamma = 1$ and
$\sigma \simeq 1$, we see that $T_{*}^{D'}$ satisfies an estimate similar to (\ref{eq:positive}). 

The next two results say, roughly speaking, that  $T_{*}^{D'}$ can be
controlled also when we are close to these two extreme cases.

\begin{lem} \label{one}
If $d' = 1$ then $T_{*}^{D'}$ is of weak type $(p_1, p_1)$.
\end{lem}

\begin{lem}\label{two}
If $d' > d''$ then $T_{*}^{D'}$ maps $L^{p_1}$ boundedly into
weak $L^{p_1}\log^{-(d'-1)\slash p_1} L$, i.e.,
$$
|\{T_{*}^{D'}f > \lambda\}| \le C \frac{ \|f\|_{p_1}^{p_1}}{\lambda^{p_1}} 
	\bigg[\log\bigg(2 + \frac{\lambda}{\|f\|_{p_1}}\bigg)\bigg]^{d'-1},
	 \qquad \lambda > 0.
$$
\end{lem}

These observations and lemmas together cover all possible choices of $D'$ for $d \leq 3$, 
so Theorem \ref{pos} follows once we prove Lemmas \ref{one} and \ref{two}. 

The possibilities for $D'$ not covered by the above are described by the inequalities 
$2\leq d' \leq d''$. From the proof of Theorem \ref{neg} given later,
it can be seen that $T^{D'}_{*}$ cannot be controlled on $L^{p_1}$ in these cases. 

In the sequel, we assume, without loss of generality, that
$D' = \{1,\ldots, d'\}$. Further, we write ${\sum}', \, {\prod}', \, {\sum}'', \, {\prod}''$ 
for sums and products taken over $1 \leq j \leq d'$ and $d' < j \leq d$, respectively.

\begin{proof}[Proof of Lemma \ref{one}]
Splitting points in $\R_+^d$ as $x = (x_1, x'')$ and suppressing the exponential 
factors in \eqref{eq:term}, we get
$$
T_{*}^{D'} f(x) \le  x_1^{-1\slash p_1}\, \|M'' f(\cdot,x'')\|_{L^{p_1}(\R_+)}.
$$
For any fixed $x''$, it is clear that the set of points $x_1$
where this expression exceeds a level $\lambda>0$ has one-dimensional measure at most 
$$
\frac{C}{\lambda^{p_1}} \|M'' f(\cdot,x'')\|_{L^{p_1}(\mathbb{R}_+)}^{p_1}.
$$
Integrating in $x''$, we conclude that 
$$
|\{T_{*}^{D'} f > \lambda\}| \le
\frac{C}{\lambda^{p_1}} \int \big[ M''f(x)\big]^{p_1}\,dx 
	\leq \frac{C}{\lambda^{p_1}} \int |f(x)|^{p_1}\,dx,
$$
which finishes the proof.
\end{proof}

\begin{proof}[Proof of Lemma \ref{two}]
Without loss of generality, we assume $0\leq f \in L^{p_1}$, with norm 1. 
Since $t \le 1$, we can replace the sum over $D''$ in \eqref{eq:term} by
$1+{\sum}'' x_j$, to get
$$
T_t^{D'} f(x) \lesssim
\exp\bigg(-\frac{c}{t} \, {\sum}' x_j - c t \, \Big(1 + {\sum}'' x_j\Big)\bigg)
\: {\prod}' x_j^{-1\slash p_1} \: \|M''f(\cdot,x'')\|_{p_1,d'},
$$
where  $\|\cdot\|_{p_1,d'}$ denotes the norm in $L^{p_1}(\mathbb{R}_+^{d'})$.
In order to eliminate $t$, we then use the inequality between the geometric and arithmetic 
means to estimate the exponential. The conclusion is
$$
T_*^{D'} f(x) \lesssim \exp\bigg(-c \sqrt{ \Big(1+{\sum}'' x_j\Big)\,{\sum}' x_j}\;\bigg)
	\: {\prod}' x_j^{-1\slash p_1} \: \|M''f(\cdot,x'')\|_{p_1,d'}.
$$

We now fix $x''$ and apply Lemma \ref{level} (c) in the $x'$ variables,
with $\gamma = 1/2$ and  $\sigma \simeq 1+{\sum}'' x_j$. Thus for $\lambda > 0$ the set of
points $x'$ where $T_{*}^{D'} f(x', x'') > \lambda$ has $d'$-dimensional measure at most
$$
\frac{C}{\lambda^{p_1}} \|M''f(\cdot,x'')\|_{p_1,d'}^{p_1} 
\bigg[\log\bigg(2 + \Big(1+{\sum}'' x_j\Big)^{-d'} \lambda^{p_1}\big\slash 
\|M''f(\cdot,x'')\|_{p_1,d'}^{p_1}\bigg)\bigg]^{d'-1}.
$$
In order to estimate the $d$-dimensional measure of the level set, 
we must integrate this quantity in $x''$. The integral over those $x''$ for which 
$$
\Big(1+{\sum}'' x_j\Big)^{-d'\slash p_1} \, 
\lambda \big\slash \|M''f(\cdot,x'')\|_{p_1,d'} < 1+\lambda
$$
is easy to handle. Indeed, here the logarithm is at most 
$\log(2 + (1+\lambda)^{p_1}) \lesssim \log (2 + \lambda)$, 
and in view of Fubini's theorem, this integral is dominated by 
$$
\frac{C}{\lambda^{p_1}} \|M''f\|_{p_1}^{p_1} 
\big[\log (2 + \lambda)\big]^{d'-1} \le \frac{C}{\lambda^{p_1}} \big[\log (2 + \lambda)\big]^{d'-1}, 
$$
since $\|f\|_{p_1} = 1.$
This agrees with the right-hand side of the inequality in Lemma \ref{two}. 

What remains is the integral where
$$
\|M''f(\cdot,x'')\|_{p_1,d'} \Big(1+{\sum}'' x_j\Big)^{d'\slash p_1} \big\slash \lambda 
	< \frac{1}{1+\lambda}.
$$
There we write the integrand as
\begin{align*}
& C \Big(1+{\sum}'' x_j \Big)^{-d'} \bigg(\|M''f(\cdot,x'')\|_{p_1,d'} 
\Big(1+{\sum}'' x_j\Big)^{d'\slash p_1}\big \slash \lambda \bigg)^{p_1} \\ 
& \times \bigg[\log\bigg(2 +\|M''f(\cdot,x'')\|_{p_1,d'}^{-p_1} \Big(1+{\sum}'' x_j\Big)^{-d'} \,
	\lambda^{p_1}\bigg)\bigg]^{d'-1}.
\end{align*}
The function $s^{p_1}[\log(2+s^{-p_1})]^{d'-1}$ is increasing for $s>0$, up to 
a constant factor. Hence, in the region considered, we can estimate the last expression by
$$
C \Big(1+{\sum}'' x_j\Big)^{-d'} (1+\lambda)^{-p_1} 
\big[\log\big(2 + (1+\lambda)^{p_1}\big)\big]^{d'-1}.
$$
Since $d' > d''$, this quantity is integrable with respect to $x''$, and the integral will be at most
$C(1+\lambda)^{-p_1}[\log(2+\lambda)]^{d'-1}$. This completes the proof.
\end{proof}

\subsection{The case of two minimal $\alpha_i$ in dimension $3$} \label{ssec:2m3} 
	\qquad \\
Now $d=3$. Without any loss of generality we may assume that $\alpha=(a,a,b)$
with $-1<a<0$ and $a<b$. The critical exponents are as in the preceding subsection.
\begin{thm} \label{pos_s}
Let $d=3$ and $\alpha$ be as above. 
Then for $f \in L^{p_1}$ the distribution function of $\M f$ satisfies
\begin{equation*}
\big|\{\M f > \lambda\}\big| \le C \frac{\|f\|_{p_1}^{p_1}}{\lambda^{p_1}} 
\log\bigg(2 + \frac{\lambda}{\|f\|_{p_1}}\bigg), \qquad \lambda > 0.
\end{equation*}
\end{thm}

To prove Theorem \ref{pos_s}, we estimate the kernel $\mathcal{H}_t^{(a,a,b)}$ by applying
the inequality \eqref{main} in each variable. 
Then a sum of $8$ terms emerges; as before, we index these terms 
by subsets $D' \subset \{1,2,3\}$ and let primed variables correspond to the second term 
in \eqref{main}.

Next, we observe that only the term corresponding to $D'=\{1,2,3\}$ requires further analysis.
Indeed, from the asymptotics \eqref{asymp} it follows immediately that
\begin{equation} \label{pqw}
\mathcal{H}_t^{b}(\xi,\eta) \lesssim \mathcal{H}_t^{a}(\xi,\eta), \qquad \xi,\eta \in \R_+, \quad 
0 < t \le 1,
\end{equation}
and therefore the cases when $d'<3$ are covered by the results of Section \ref{ssec:amp1},
see Lemmas \ref{one} and \ref{two}.

Assume then that $d'=3$. The kernel under consideration is
$$
H_t(x,y) = \frac{(x_1 x_2 y_1 y_2)^{a\slash 2}}{t^{2(a+1)}} \frac{(x_3 y_3)^{b\slash 2}}{t^{b+1}}
	\exp\Big( - c \sum_{j=1}^3 \frac{x_j + y_j}{t} \Big),
$$
and we are interested in the maximal operator 
$H_*f(x) = \sup_{0 < t \le 1} \int H_t(x,y) |f(y)|\, dy$. 
Theorem \ref{pos_s} will be proved once we verify the following.
\begin{lem} \label{lpos_s}
For $f \in L^{p_1}$ the distribution function of $H_* f$ satisfies
\begin{equation*}
\big|\{H_* f > \lambda\}\big| \le C \frac{ \|f\|_{p_1}^{p_1}}{\lambda^{p_1}} 
\log\bigg(2 + \frac{\lambda}{\|f\|_{p_1}}\bigg), \qquad \lambda > 0.
\end{equation*}
\end{lem} 

\begin{proof}
Observe that
$$
H_* f(x) \simeq \sup_{k\ge 0} \int H_{2^{-k}}(x,y) |f(y)|\, dy 
	\le \sum_{k\ge 0} \int H_{2^{-k}}(x,y) |f(y)| \, dy.
$$
Splitting now the kernel in the third variable according to the dyadic intervals
$2^{-k-\nu} < x_3 \le 2^{-k-\nu+1}$ and $2^{-k-\beta} < y_3 \le 2^{-k-\beta+1}$,
$\nu, \beta \in \Z$ (written shortly $x_3 \sim 2^{-k-\nu}$ and $y_3 \sim 2^{-k-\beta}$),
we get
$$
H_*f(x) \lesssim \sum_{\nu,\beta \in \Z} \sum_{k\ge 0} H_{k,\nu,\beta}f(x), \qquad x \in \R^3_+,
$$
where
$$
H_{k,\nu,\beta}f(x) = \int_{\R^3_+} H_{2^{-k}}(x,y) 
	\chi_{\{x_3 \sim 2^{-k-\nu},\; y_3 \sim 2^{-k-\beta}\}} |f(y)|\, dy.
$$
Since
\begin{align} \nonumber
H_{k,\nu,\beta}f(x) \lesssim & \,2^{(2a+3)k}2^{-(\nu+\beta)b\slash 2} 
	\exp\big( -c(2^{-\nu}+2^{-\beta})\big) \chi_{\{x_3 \sim 2^{-k-\nu}\}} (x_1 x_2)^{a\slash 2}
		\\ & \times \exp\big( -c 2^k (x_1+x_2) \big) \int_{y_3 \sim 2^{-k-\beta}} (y_1 y_2)^{a\slash 2}
			\exp\big( -c 2^k (y_1+y_2) \big) |f(y)| \, dy, \label{bs}
\end{align}
an application of H\"older's inequality leads to the estimate (recall that $a = -2\slash p_1$)
\begin{align*}
H_{k,\nu,\beta}f(x) \lesssim & \, 2^{3k} 2^{-4k \slash p_1} 2^{-(\nu+\beta)b \slash 2}
	\exp\big( -c(2^{-\nu}+2^{-\beta}) \big) 2^{-(k+\beta)\slash p_0} 2^{2k\slash p_1 - 2k \slash p_0}
	(x_1 x_2)^{-1\slash p_1} \\ & \times 
	\exp\big( - c 2^k (x_1+x_2) \big) \chi_{\{x_3 \sim 2^{-k-\nu}\}}
		\bigg( \int_{y_3 \sim 2^{-k-\beta}} |f(y)|^{p_1} dy \bigg)^{1\slash p_1}.
\end{align*}
A short computation shows that the constant factor in the last expression equals
$$
2^{k \slash p_1} 2^{\nu\slash p_1} 2^{-\delta \beta} 2^{-(\nu+\beta)\varepsilon}
	\exp\big( -c(2^{-\nu}+2^{-\beta})\big),
$$
where $\delta = 1\slash p_0-1\slash p_1 >0$ and $\varepsilon = (b-a)\slash 2 >0$. Therefore
\begin{align*}
H_{k,\nu,\beta}f(x) \lesssim & \, 2^{k\slash p_1} 2^{\nu \slash p_1} 2^{-\varepsilon(|\nu|+|\beta|)}
	(x_1 x_2)^{-1\slash p_1} \exp\big( - c 2^k (x_1+x_2) \big) \chi_{\{x_3 \sim 2^{-k-\nu}\}}
		\\ & \times \bigg( \int_{y_3 \sim 2^{-k-\beta}} |f(y)|^{p_1} dy \bigg)^{1\slash p_1}.
\end{align*}

Consequently, we see that the condition
$H_{k,\nu,\beta} f(x) > c \lambda 2^{-\varepsilon (|\nu|+|\beta|)\slash 2}$ implies
$$
\chi_{\{x_3 \sim 2^{-k-\nu}\}} (x_1 x_2)^{-1} \exp\big( - c 2^k (x_1+x_2) \big)
	\gtrsim 2^{-k-\nu} 2^{\varepsilon (|\nu|+|\beta|)p_1\slash 2}
		\bigg( \int_{y_3 \sim 2^{-k-\beta}} \!\!|f(y)|^{p_1} dy \bigg)^{-1} \lambda^{p_1}.
$$
Applying now Lemma \ref{level} (c) with $d=2$ in the first two variables, we get
\begin{align*}
& \big|\big\{x \in \R^3_+ : H_{k,\nu,\beta}f(x) > c \lambda 2^{-\varepsilon (|\nu|+|\beta|)\slash 2} \big\}\big|
\lesssim 2^{-\varepsilon (|\nu|+|\beta|)p_1\slash 2} 
\frac{1}{\lambda^{p_1}} \int_{y_3 \sim 2^{-k-\beta}}
	|f(y)|^{p_1} dy  \\ &
 \quad \times \log\bigg[ 2+ 2^{-2k} 2^{-k-\nu} 2^{\varepsilon (|\nu|+|\beta|)p_1\slash 2}
	\lambda^{p_1} \bigg( \int_{y_3 \sim 2^{-k-\beta}} |f(y)|^{p_1} dy \bigg)^{-1} \bigg];
\end{align*}
notice that the logarithm here is at most a constant times
$$
\Phi_{k,\nu,\beta}(\lambda) = |\nu| + |\beta| + \log\bigg[ 2+ 2^{-3k} 
	\lambda^{p_1} \bigg( \int_{y_3 \sim 2^{-k-\beta}} |f(y)|^{p_1} dy \bigg)^{-1} \bigg].
$$
Since the above level sets are disjoint for different $k$, it follows that
\begin{align}
& \Big|\Big\{x \in \R^3_+ : \sum_{k\ge 0} H_{k,\nu,\beta}f(x) > c 
	\lambda 2^{-\varepsilon (|\nu|+|\beta|)\slash 2} \Big\}\Big|  = 
	\sum_{k\ge 0} \big|\{H_{k,\nu,\beta}f(x) > c 
	\lambda 2^{-\varepsilon (|\nu|+|\beta|)\slash 2} \}\big| \nonumber \\
& \lesssim 2^{-\varepsilon (|\nu|+|\beta|)p_1\slash 2} 
	\sum_{k\ge 0} \frac{1}{\lambda^{p_1}}
	\int_{y_3 \sim 2^{-k-\beta}} |f(y)|^{p_1} dy \; \Phi_{k,\nu,\beta}(\lambda). \label{cre}
\end{align}

To estimate the right-hand side here, we start by observing that 
$$
 2^{-\varepsilon (|\nu|+|\beta|)p_1\slash 2} \sum_{k\ge 0}
	\frac{1}{\lambda^{p_1}}
	\int_{y_3 \sim 2^{-k-\beta}} |f(y)|^{p_1} dy  \; \big(|\nu|+|\beta|\big) \lesssim 
	 2^{-\varepsilon (|\nu|+|\beta|) \slash 2} \frac{\|f\|_{p_1}^{p_1}}{\lambda^{p_1}}.
$$
The remaining part of the right-hand side in \eqref{cre} is
$$
2^{-\varepsilon (|\nu|+|\beta|)p_1\slash 2} \sum_{k\ge 0} \frac{1}{\lambda^{p_1}}
	\int_{y_3 \sim 2^{-k-\beta}} |f(y)|^{p_1} dy \,
	\log\bigg[ 2+ 2^{-3k} 
	\lambda^{p_1} \bigg( \int_{y_3 \sim 2^{-k-\beta}} |f(y)|^{p_1} dy \bigg)^{-1} \bigg].
$$
To estimate the sum here, we consider two cases. If
\begin{equation} \label{crst}
\int_{y_3 \sim 2^{-k-\beta}} |f(y)|^{p_1} dy \ge 2^{-3k} \|f\|_{p_1}^{p_1},
\end{equation}
the argument of the last logarithm is at most $2 + (\lambda \slash \|f\|_{p_1})^{p_1}$.
So summing the terms with this property in the above sum, we get at most
$$
\sum_{k\ge 0} \frac{1}{\lambda^{p_1}} \int_{y_3 \sim 2^{-k-\beta}} |f(y)|^{p_1} dy \,
	\log\bigg( 2 + \frac{\lambda^{p_1}}{\|f\|_{p_1}^{p_1}} \bigg) 
		\le \frac{\|f\|_{p_1}^{p_1}}{\lambda^{p_1}} 
			\log\bigg( 2 + \frac{\lambda^{p_1}}{\|f\|_{p_1}^{p_1}} \bigg).
$$
For the terms not satisfying \eqref{crst}, we use the essential
monotonicity of the function $s \log(2+s^{-1})$ to estimate that part of the sum by
$$
\frac{C}{\lambda^{p_1}} \sum_{k\ge 0} 2^{-3k}
	\|f\|_{p_1}^{p_1}
	\log\bigg( 2 + \frac{\lambda^{p_1}}{\|f\|_{p_1}^{p_1}} \bigg) \simeq
		\frac{\|f\|_{p_1}^{p_1}}{\lambda^{p_1}} 
			\log\bigg( 2 + \frac{\lambda^{p_1}}{\|f\|_{p_1}^{p_1}} \bigg).
$$

Altogether, this gives
$$
\Big|\Big\{x \in \R^3_+ : \sum_{k\ge 0} H_{k,\nu,\beta}f(x) > c 
	\lambda 2^{-\varepsilon (|\nu|+|\beta|)\slash 2} \Big\}\Big| \lesssim
		2^{-\varepsilon (|\nu|+|\beta|)\slash 2} \frac{\|f\|_{p_1}^{p_1}}{\lambda^{p_1}} 
			\log\bigg( 2 + \frac{\lambda}{\|f\|_{p_1}} \bigg).
$$
Now the exponentially decreasing factor allows us to sum these estimates in $\nu$ and $\beta$,
and we finally conclude that
\begin{align*}
|\{x \in \R^3_+ : H_{*}f(x) > \lambda\}| & \le \sum_{\nu,\beta \in \Z}
\Big|\Big\{x \in \R^3_+ : \sum_{k\ge 0} H_{k,\nu,\beta}f(x) > c 
	\lambda 2^{-\varepsilon (|\nu|+|\beta|)\slash 2} \Big\}\Big| \\
& \lesssim \frac{\|f\|_{p_1}^{p_1}}{\lambda^{p_1}} 
			\log\bigg( 2 + \frac{\lambda}{\|f\|_{p_1}} \bigg).
\end{align*}
This finishes the proof of Lemma \ref{lpos_s}.
\end{proof}

\subsection{Counterexamples} \label{cex_p1} \qquad \\
Assume now that $d\ge 4$ and $\alpha$ is such that $\widetilde{\alpha}<0$ and 
$\widetilde{d}(\alpha)\ge 2$.
We shall construct functions proving the negative part of Theorem
\ref{main_thm} (b2). Here we may replace $T_*^{\alpha}$ by $\M$, since only 
$t \le 1$ will be considered. For the sake of clarity, we state the result separately.

\begin{thm} \label{neg}
For $d \geq 4$ and $\alpha$ as above, there exists a function 
$f \in L^{p_1,1}$ such that
$$
\big|\{\M f > \lambda\}\big| = \infty, \qquad  \lambda > 0.
$$
\end{thm}

We will prove the theorem in the case when all $\alpha_i$ are minimal.
The same reasoning works in the general case, as seen by including the variables 
corresponding to non-minimal $\alpha_i$ among the double-primed variables below.

\begin{proof}[Proof of Theorem \ref{neg}]
We continue to use the splitting $\R_+^d = \R_+^{d'} \times \R_+^{d''}$ and the related 
notation. Assuming to begin with that $d \geq 5$, we can choose $d'$ so that $2 \leq d' < d''$.
We shall then construct an $L^{p_1,1}$ function as in the statement of Theorem \ref{neg}. 
The same function will actually show that the corresponding operator $T_{*}^{D'}$ cannot be
controlled on $L^{p_1,1}$.

Let for small $t>0$ the set $E_t \subset \R^d_+$ be defined by $t < y_j < 2t$ for $j \leq d'$ 
and $t^{-1}<y_j <2t^{-1}$ for $d' <j \leq d$. Let $f_t = t^{(d''-d')\slash p_1} \chi_{E_t}$, 
which has $L^{p_1,1}$ norm essentially $1$. Clearly,
$$
\M f_t(x) \geq t^{(d''-d')\slash p_1} \int_{E_t}\mathcal{H}_t^{\alpha}(x,y)\,dy,
$$
and here we take points $x$ with $x_j < t$ for $j \leq d'$ and 
$t^{-1}< x_j <2 t^{-1}$ for $d' <j \leq d$. With such an $x$, we further restrict the 
integration above by the condition $|y_j -x_j| < 1$ for $d' <j \leq d$. For such $x$ and $y$, 
Lemma \ref{lower}, part (b) for the first $d'$ variables and part (a) for the remaining ones, 
implies that the kernel $\mathcal{H}_t^{\alpha}(x,y)$ 
is at least $c t^{d'\slash p_1-d'} {\prod}' x_j^{-1\slash p_1}$. We get
$$
\M f_t(x) \gtrsim t^{d''\slash p_1} {\prod}' x_j^{-1\slash p_1}.
$$

We now apply Lemma \ref{level} (b) in dimension $d'$, with 
$\nu = t^{-d''}\lambda^{p_1}$ for some $\lambda>0$ and $t$ so small that 
$t^{d'-d''} \lambda^{p_1} \ge 1$.
For each fixed $x''$, we conclude that the $d'$-dimensional measure
of the level set $\{x' : \M f_t(x',x'') > \lambda\}$ is at least
$ct^{d''}\lambda^{-p_1} [\log(t^{d'-d''}\lambda^{p_1})]^{d'-1}$. Integrating
with respect to $x''$ in the set where $t^{-1}<x_j < 2t^{-1}$ for $j>d'$, 
we see that the $d$-dimensional measure of $\{\M f_t > \lambda\}$ is at least
$c\lambda^{-p_1} [\log (t^{d'-d''}\lambda^{p_1})]^{d'-1}$.

But we can make the last quantity arbitrarily large by taking
$t$ small, for any fixed $\lambda>0$. This shows that the condition $f\in L^{p_1,1}$
gives no control of the level set. By taking linear combinations
of such $f_t$, it is easy to construct an $L^{p_1,1}$ function such that
all the level sets of $\M f$ have infinite measure.

To cover also the case $d=4$, we now consider $d'$ with $2 \le d' = d''=d\slash 2$.
Here the construction is a bit more subtle. We shall take essentially
the characteristic function of a union of sets like $E_t$ with $1\slash R < t < 1\slash 4$, 
for large values of $R$. More precisely, for $R>6$ we define the set
\begin{align*}
E_R = \big\{y \in \mathbb{R}_+^d : \;& 1 < y_d < R, \;\; y_d^{-1}/4 < y_j < 2y_d^{-1} 
\;\; \mathrm{for} \;\; 1 \leq j \leq d' \\ 
& \mathrm{and}\;\; y_d/8 < y_j < 8y_d \;\; \mathrm{for} \;\; d' < j < d\big\}.
\end{align*}
Then
$$
|E_R| \simeq \int_1^R y_d^{-d'+d''-1}\,dy_d = \log R,
$$
and we define the function $f_R = |E_R|^{-1\slash p_1} \chi_{E_R}$, whose norm in 
$L^{p_1,1}$ is essentially $1$. We shall estimate $\M f_R(x)$ at points $x$ with 
$4 < x_d < R-1$ and $0 < x_j < x_d^{-1}$ 
for  $1 \leq j \leq d'$ and $x_d/2 < x_j < 2x_d$ for $d' < j < d$. Then
$$
\M f_R(x) \gtrsim (\log R)^{-1\slash p_1} \int_{E_R}\mathcal{H}_t^{\alpha}(x, y)\,dy, 
$$ 
where we choose $t = x_d^{-1}$. Further, we restrict this integral to the set
\begin{align*}
F_x = \big\{ y \in \R^d_+ : 
	x_d^{-1}/2 < y_j < x_d^{-1}\;\; \mathrm{for} \;\; 1 \leq j \leq d' \;\; 
 \mathrm{and}  \;\;   |y_j - x_j| < 1 \;\; \mathrm{for} \;\; d' < j \leq d \big\};
\end{align*}
some simple computations show that $F_x \subset E_R$, if $x$ is as described above. 
For $y \in F_x$, items (a) and (b) of Lemma \ref{lower} then imply
$$
\mathcal{H}_{x_d^{-1}}^{\alpha}(x, y) \gtrsim x_d^{d'-d'\slash p_1} \, {\prod}' x_j^{-1\slash p_1}.
$$
Integrating in $y$ over $F_x$, we conclude that
\begin{equation} \label{eq:11}
	\M f_R(x) \gtrsim (\log R)^{-1\slash p_1}\, x_d^{-d'\slash p_1} \, {\prod}' x_j^{-1\slash p_1}.
\end{equation}
Now fix a point $x'' \in \R^{d''}_{+}$ with $4 < x_d <  R-1$ and
$x_d/2 < x_j < 2x_d$ for $d' < j <d$. Then if $x' \in (0, x_d^{-1})^{d'}$ satisfies 
${\prod}' x_j^{-1\slash p_1} > (\log R)^{1\slash p_1} x_d^{d'\slash p_1}\lambda$ for 
some $\lambda > 0$, \eqref{eq:11} implies that $\M f_R(x', x'') \gtrsim \lambda$. 
In view of Lemma \ref{level} (b), under the assumption $\lambda > (\log R)^{-1\slash p_1}$
the set of such $x'$ has $d'$-dimensional measure at least 
$$
c (\log R)^{-1} x_d^{-d'} \lambda^{-p_1} \big[\log\big(2 + x_d^{-d'}(\log R)
x_d^{d'}\lambda^{p_1}\big)\big]^{d'-1}.
$$
Now we integrate in $x''$, over the set specified above. We conclude that
the $d$-dimensional measure of the set where $\M f_R(x) > c\lambda$ is at least
$$
c (\log R)^{-1} \lambda^{-p_1} \big[\log(2 +\lambda^{p_1}\log R)\big]^{d'-1} \int_4
^{R-1} x_d^{-1}\,dx_d
\simeq \lambda^{-p_1} \big[\log(2 +\lambda^{p_1} \log R)\big]^{d'-1}.
$$
Since $f_R$ is normalized in $L^{p_1,1}$ and $R$ can be chosen arbitrarily
large, we get the same conclusions as in the case $2 \leq d' <d''$.
\end{proof}

\subsection{Comment on sharpness} \label{sharp_p1}

In Theorem \ref{main_thm} (a2) and (b2) the weak-type space
$L^{p_1,\infty}\log^{-(\widetilde{d}(\alpha)-1)\slash p_1} L$
is sharp in the following sense. There exists a function $f$, not only in $L^{p_1}$ 
but bounded and of compact support, such that for large $\lambda$,
$$
|\{ T^{\alpha}_* f > \lambda \}| \simeq 
	\lambda^{-p_1} \big[ \log(2+\lambda) \big]^{\widetilde{d}(\alpha)-1}.
$$
This $f$ can simply be chosen as the characteristic function of the cube $(1\slash 2,1)^d$.

Indeed, in the case $\widetilde{d}(\alpha)=d$, that is when all $\alpha_i$ are minimal,
Lemma \ref{lower} (b) implies
$$
T^{\alpha}_1 f(x) \simeq \int \mathcal{H}^{\alpha}_{1\slash 2}(x,y) f(y)\, dy
\gtrsim \prod_{j=1}^d x_j^{-1\slash p_1}, \qquad x \in (0,1)^d.
$$
Since $T^{\alpha}_*f \ge T^{\alpha}_1 f$, we see from Lemma \ref{level} (b) that
the level sets of $T^{\alpha}_*f$ are as claimed. In the general case we again use 
Lemma \ref{lower} (b) to estimate $T^{\alpha}_1 f(x)$ from below by a suitable product. 
Then an application of Lemma \ref{level} (b) in the variables corresponding to the minimal
$\alpha_i$ and integration in the remaining variables  lead to the conclusion.

This observation shows, in particular, that $T^{\alpha}_*$
is not of strong type $(p_1,p_1)$, even if there is only one minimal $\alpha_i$.

\section{The endpoint $p_0$} \label{sec:p0}

We keep the notation introduced in the previous sections. The operator under consideration 
is still $\M$ rather than $T_*^{\alpha}$. We first deal with the situation when there is only one 
minimal value $\alpha_i$. In this case the restricted weak type $(p_0,p_0)$ follows quickly 
from the results of Section \ref{sec:strong}. Indeed, assume that $d \ge 2$ and $\alpha_1$ 
is the only minimal $\alpha_i$. Then the maximal operator
$$
\mathcal{H}_*^{(\alpha_2,\ldots,\alpha_d)} f(x) = \sup_{t>0} \int
	\mathcal{H}_t^{(\alpha_2,\ldots,\alpha_d)}
		\big( (x_2,\ldots,x_d),(y_2,\ldots,y_d)\big) |f(x_1,y_2,\ldots,y_d)| \, 
			dy_2\cdots dy_d
$$
is bounded on $L^p(\R^d_+)$ for $p$ in an interval strictly containing the point
$p_0 = p_0(\widetilde{\alpha})=p_0(\alpha_1)$. By interpolation, see for instance 
\cite[Theorem 4.13]{BeSh}, it is then also bounded on the Lorentz space $L^{p_0,1}(\R^d_+)$.
Moreover, the one-dimensional maximal operator $\mathcal{H}_*^{\alpha_1}$ satisfies
the restricted weak-type $(p_0,p_0)$ estimate (this was already proved in Section 
\ref{sec:strong}) and the same is true for its $d$-dimensional extension
$$
\mathcal{H}_*^{\alpha_1} f(x) = \sup_{t>0} \int \mathcal{H}_*^{\alpha_1}(x_1,y_1)
	|f(y_1,x_2,\ldots,x_d)| \, dy_1,
$$
as easily verified. Since restricted weak type $(p_0,p_0)$ means boundedness from 
$L^{p_0,1}$ to weak $L^{p_0}$ and
$$
\M f(x) \le \mathcal{H}_*^{\alpha_1} \circ
	\mathcal{H}_*^{(\alpha_2,\ldots,\alpha_d)} f(x), \qquad x \in \R^d_+,
$$
item (a3) in Theorem \ref{main_thm} follows.

Proving the remaining results is less straightforward. As in Section 
\ref{sec:p1}, we consider two main cases: when all $\alpha_i$ are minimal and when
there are precisely two minimal $\alpha_i$ in dimension $3$. These two cases will
justify the estimate of Theorem \ref{main_thm} (b3). Later we will construct
counterexamples disproving similar estimates in dimensions $d\ge 4$.

\subsection{The case when all $\alpha_i$ are minimal} \label{ssec:allmin} \qquad \\ 
We work in dimension $d \ge 2$. All the $\alpha_i$ are assumed to be $a$, with $-1<a<0$.
The critical exponents are $p_1=-2\slash a$ and $p_0 = p'_1 = 2\slash (a+2)$.

\begin{thm} \label{pos_p0}
For $d=2,3$ and $\alpha$ as above, the operator $\M$ maps $L^{p_0,1}\log^{(d-1)\slash p_1}L$ 
into $L^{p_0,\infty}$, in the sense that for all $E \subset \R^d_+$ of finite measure 
$$
\big|\{\M \chi_E > \lambda \}\big| \le C \frac{|E|}{\lambda^{p_0}}
								\bigg[ \log\bigg(2+ \frac{1}{|E|}\bigg) 
									\bigg]^{\frac{p_0}{p_1}({d}-1)}, \qquad \lambda > 0.
$$
\end{thm}

We shall prove this theorem by applying the bound \eqref{main} to estimate
$\int \mathcal{H}_t^{\alpha}(x,y) |f(y)|\, dy$. As in the proof of Theorem \ref{pos}, this
will produce $2^d$ terms, indexed again by subsets $D'$ of $\{1,\ldots,d\}$.
We use all the notation from that proof, letting now the $x'$ variables correspond to
the second term in \eqref{main}. 

Thus the kernel is controlled by the sum in $D'$ of
\begin{align*}
& t^{(2\slash p_1 -1)d'} {\prod}' x_j^{-1\slash p_1} {\prod}' y_j^{-1\slash p_1}
\exp\Big( -\frac{c}{t}\Big( {\sum}' x_j + {\sum}' y_j \Big) \Big) \\
& \times \exp\Big( -{c}{t}\Big( {\sum}'' x_j + {\sum}'' y_j \Big) \Big)
{\prod}'' \frac{1}{\sqrt{t x_j}} \exp\Big( -c \frac{(y_j-x_j)^2}{t x_j} \Big).
\end{align*}
Observe that
$$
\exp\Big( -\frac{c}{t} {\sum}' x_j \Big) \lesssim 
	\Big( \frac{1}{t} {\sum}' x_j \Big)^{(2\slash p_1 -1)d'}
$$
and
$$
\int {\prod}'' \frac{1}{\sqrt{t x_j}} \exp\Big( -c \frac{(y_j-x_j)^2}{t x_j} \Big)
	|f(y',y'')|\, dy''  \lesssim M'' f(y',x'').
$$
As in the proof of Lemma \ref{two}, we use
the inequality between arithmetic and geometric means to conclude that
$$
\int \mathcal{H}_t^{\alpha}(x,y) |f(y)|\, dy \lesssim \sum_{D'} S^{D'}f(x), 
$$
where
\begin{align*}
 S^{D'}f(x) & =  {\prod}' x_j^{-1\slash p_1} \Big({\sum}' x_j\Big)^{(2\slash p_1 -1)d'} \\
&	\quad \times \int {\prod}' y_j^{-1\slash p_1} 
	\exp\bigg(-c\sqrt{ \Big(1+{\sum}'' x_j\Big)\,{\sum}' y_j}\;\bigg) M'' f(y',x'')\, dy'
\end{align*}
does not depend on $t$. So it suffices to obtain suitable estimates for each operator $S^{D'}$.

The case when $D' = \emptyset$ (i.e. $d'=0$) is simple, since then $S^{D'}$ is obviously
bounded on $L^{p_0}$. The remaining cases are treated in the lemmas below.

\begin{lem} \label{tres}
If $d'=d$ then for all $E \subset \R^d_+$ of finite measure
$$
\big|\big\{S^{D'}\chi_{E}> \lambda\big\}\big| 
\le C \frac{|E|}{\lambda^{p_0}}\bigg[ \log\Big( 2+\frac{1}{|E|}\Big)
	\bigg]^{\frac{p_0}{p_1}(d'-1)}, \qquad \lambda > 0.
$$
\end{lem}

\begin{lem} \label{dos}
The estimate of Lemma \ref{tres} is true whenever $d'>d''\ge 1$.
\end{lem}

\begin{lem} \label{uno}
If $d'=1$ then $S^{D'}$ is of restricted weak type $(p_0,p_0)$.
\end{lem}

These results together cover all possible choices of $D'$ for $d \le 3$, so Theorem \ref{pos_p0}
follows once we prove Lemmas \ref{tres}-\ref{uno}. The remaining possibilities for $D'$ are
described by the inequalities $2 \le d' \le d''$. In the proof of Theorem \ref{neg_p0} 
given later, we shall see that $S^{D'}$ then cannot be controlled in a similar manner.

For the proofs of Lemmas \ref{tres}-\ref{uno}, we need two more preparatory results. We introduce 
a notation for the product of the first two factors in the expression for $S^{D'}f$, defining
$$
\psi_d(x) = \bigg(\sum_{j=1}^d x_j\bigg)^{({2}\slash{p_1}-1)d} \; \prod_{j=1}^d x_j^{-1\slash p_1},
	\qquad x \in \R^d_+.
$$

\begin{lem} \label{psi}
The function $\psi_d$ belongs to $L^{p_0,\infty}(\R^d_+)$.
\end{lem}

\begin{proof}
We may assume for symmetry reasons that $x_1 = \max_{1\le j \le d} x_j$. 
When $x_j \simeq x_1$ for all $j$, one easily finds that
$\psi_d(x) \simeq x_1^{-d\slash p_0} \simeq |x|^{-d\slash p_0}$,
and the function $x \mapsto |x|^{-d\slash p_0}$ is in $L^{p_0,\infty}(\R^d_+)$.
And when $x$ is in the sector $S_k$ defined by
$2^{-k_j-1} < x_j\slash x_1 \le 2^{-k_j}$, for $k_j \ge 0$, $j=2,\ldots,d$,
$$
\psi_d(x) \simeq \bigg( \prod_{j=2}^d 2^{k_j\slash p_1} \bigg) x_1^{-d\slash p_0}
	\simeq 2^{\sum k_j \slash p_1} |x|^{-d\slash p_0}.
$$
The sector $S_k$ has aperture comparable to $2^{-\sum k_j}$, and it is easy to check that 
$$
\lambda^{p_0} |\{\chi_{S_k}\psi_d > \lambda\}|
\lesssim 2^{-\sum k_j} \, 2^{\sum k_j p_0 \slash p_1},
$$
uniformly in $\lambda>0$ and in the $k_j$. Since $p_0<p_1$, these estimates can be summed over 
$(k_2,\ldots,k_d) \in \N^{d-1}$ to give the desired conclusion.
\end{proof}

The lemma below provides an inequality for decreasing rearrangements related to
Lemma \ref{level} (c).

\begin{lem} \label{drear}
Let $\gamma,\sigma >0$. The function
$$
F_{\sigma}(x) =
	\bigg( \prod_{j=1}^d x_j \bigg)^{-1\slash p_1} \exp\bigg( -
	\Big( \sigma \sum_{j=1}^d x_j \Big)^{\gamma} \bigg), \qquad x \in \R^d_+,
$$
has a decreasing rearrangement which satisfies
$$	
F_{\sigma}^*(s) \le C_{\gamma} \,
		s^{-1\slash p_1} 
		\bigg[\log\bigg( 2+  \frac{1}{\sigma^{d}s} \bigg)\bigg]^{(d-1)\slash {p_1}},
		\qquad s > 0,
$$
for some $C_{\gamma} < \infty$.
\end{lem}

\begin{proof}
Since
$$
F_{\sigma}^{*}(s) = \inf\big\{ \lambda > 0 : |\{x : F_{\sigma}(x)>\lambda\}| 
	\le s \big\}, \qquad s >0,
$$
we need only verify that $|\{x : F_{\sigma}(x)>\lambda\}| \le s$ for
$$
\lambda = C_{\gamma} \,
		s^{-1\slash p_1} 
		\bigg[\log\Big( 2+  \frac{1}{\sigma^{d}s} \Big)\bigg]^{(d-1)\slash {p_1}}
$$
with some suitably large $C_{\gamma}$. But Lemma \ref{level} (c) implies that
$$
|\{x : F_{\sigma}(x)>\lambda\}| \le \widetilde{C}_{\gamma} \lambda^{-p_1}
	\bigg[ \log\Big( 2+ \frac{\lambda^{p_1}}{\sigma^d} \Big) \bigg]^{d-1}
$$
for $\lambda >0$ and some $\widetilde{C}_{\gamma}$.
For the value of $\lambda$ just indicated, we thus get
\begin{align*}
& |\{x : F_{\sigma}(x)>\lambda\}| \\
& \le \widetilde{C}_{\gamma} \, C_{\gamma}^{-p_1}
	s \bigg[\log\Big( 2+  \frac{1}{\sigma^{d}s} \Big)\bigg]^{1-d}
		\bigg[ \log\bigg( 2 + C_{\gamma}^{p_1} \frac{1}{\sigma^d s}
			\bigg[\log\Big( 2+  \frac{1}{\sigma^{d}s} \Big)\bigg]^{d-1}
		\bigg)\bigg]^{d-1}.
\end{align*}
Applying the elementary inequality
$$
\log(2+xy) < \log(2 + x) + \log y, 
	\qquad x>0, \quad y > 1,
$$
to the second logarithm above, we conclude
$$
|\{x : F_{\sigma}(x)>\lambda\}| \le 
 \widetilde{C}_{\gamma} \, C_{\gamma}^{-p_1} s
	\left( \frac{ \log\left( 2 + \frac{1}{\sigma^{d}s} \right)
		+ \log C_{\gamma}^{p_1} 
		+ (d-1)\log \log\left( 2+  \frac{1}{\sigma^{d}s} \right)}
	{\log\left( 2+  \frac{1}{\sigma^{d}s} \right)}	
	\right)^{d-1}.
$$
The right-hand side here will clearly be less than $s$ if we choose $C_{\gamma}$
large enough, uniformly in $s$, which finishes the proof.
\end{proof}

\begin{proof}[Proof of Lemma \ref{tres}]
Now $d'=d$ and $S^{D'}\chi_E$ has the form
$$
S^{D'}\chi_E(x) = \psi_d(x) \int \prod y_j^{-1\slash p_1} 
	\exp\Big( -\sqrt{c \sum y_j} \, \Big) \chi_E(y)\, dy.
$$
With the aid of Lemma \ref{drear} taken with $\sigma =c$ and $\gamma=1\slash 2$, we get
\begin{align*}
S^{D'}\chi_E(x) & \le
\psi_d(x) \int_0^{|E|} 
	F_{\sigma}^*(s) \, ds \\
& \lesssim
	\psi_d(x) \int_0^{|E|} s^{-1\slash p_1} \bigg[\log\Big(2 + 
	\frac{1}{s}\Big)\bigg]^{(d-1)\slash p_1} \, ds \\
& \simeq \psi_d(x) |E|^{1\slash p_0} 
	\bigg[\log\Big(2 + \frac{1}{|E|}\Big)\bigg]^{(d-1)\slash p_1}.
\end{align*}
Lemma \ref{psi} then implies
$$
\big|\big\{S^{D'}\chi_E > \lambda\big\}\big| \lesssim \frac{|E|}{\lambda^{p_0}}
 \bigg[\log\Big(2+\frac{1}{|E|} \Big)\bigg]^{\frac{p_0}{p_1}(d-1)}, \qquad \lambda >0,
$$
as desired.
\end{proof}

\begin{proof}[Proof of Lemma \ref{dos}]
Let for $E \subset \R^d_+$ of finite measure
\begin{align*}
U^{D'}\chi_E(x) & = \big[ \psi_{d'}(x')\big]^{-1} S^{D'}\chi_E(x) \\
& =  \int {\prod}' y_j^{-1\slash p_1} 
	\exp\bigg(-c\sqrt{\Big(1+{\sum}'' x_j\Big)\, {\sum}' y_j}\;\bigg) 
		M'' \chi_E(y',x'')\, dy'.
\end{align*}
Here
$$
M'' \chi_E \le \sum_{k\ge 1} 2^{-k} \chi_{E_k},
$$
where
$$
E_k = \{M''\chi_E > 2^{-k}\} \subset \R^d_+.
$$
Given any $\varepsilon >0$, the operator $M''$ is of strong and hence weak type
$(1+\varepsilon,1+\varepsilon)$ in $\R^d_+$, so that
\begin{equation} \label{iEk}
|E_k| \lesssim \frac{1}{2^{-k(1+\varepsilon)}} \int \chi_{E}^{1+\varepsilon} 
	= 2^{(1+\varepsilon)k} |E|.
\end{equation}
With $x''$ fixed we define the slice
$$
E_k^{x''} = \{x': (x',x'') \in E_k\} \subset \R^{d'}_+.
$$
Observe that
\begin{equation} \label{dec}
U^{D'}\chi_E \lesssim \sum_{k \ge 1} 2^{-k} U_k^{D'} \chi_E,
\end{equation}
where
$$
U_k^{D'} \chi_E(x'') = \int_{E_k^{x''}} {\prod}' y_j^{-1\slash p_1}
	\exp\bigg(-c\sqrt{\Big(1+{\sum}'' x_j\Big)\, {\sum}' y_j }\;\bigg) \, dy'.
$$
Applying Lemma \ref{drear} in dimension $d'$, with $\sigma \simeq 1+{\sum}'' x_j$, we get
\begin{align*}
U_k^{D'} \chi_E(x'') 
& \lesssim \int_0^{|E_k^{x''}|} s^{-1\slash p_1} \bigg[\log\bigg( 
	2+ \Big(1+{\sum}'' x_j\Big)^{-d'} s^{-1}\bigg)\bigg]^{(d'-1)\slash p_1} ds\\
& \simeq |E_k^{x''}|^{1\slash p_0} \bigg[\log\bigg( 
	2+ \Big(1+{\sum}'' x_j\Big)^{-d'} {|E_k^{x''}|^{-1}}\bigg)\bigg]^{(d'-1)\slash p_1}.
\end{align*}

Since by Lemma \ref{psi} applied in $\R^{d'}_+$
$$
\big|\big\{x' \in \R_+^{d'} : \psi_{d'}(x') U^{D'}_k \chi_E(x'') > \lambda\big\}\big|
	 \lesssim \frac{1}{\lambda^{p_0}} \big( U^{D'}_{k}\chi_E(x'') \big)^{p_0},
$$
this implies
\begin{align*}
&\big|\big\{x : \psi_{d'}(x') U^{D'}_k \chi_E(x'') > \lambda\big\}\big| \\
&	\lesssim \frac{1}{\lambda^{p_0}}
	\int_{\R^{d''}_+} |E_k^{x''}| \bigg[\log\bigg( 2+ \Big( 
	1+ {\sum}'' x_j\Big)^{-d'} |E_k^{x''}|^{-1}\bigg)\bigg]^{\frac{p_0}{p_1}(d'-1)}  dx''.
\end{align*}
We claim that the right-hand side above is controlled, in the sense of $\lesssim$, by
$$
\frac{|E_k|}{\lambda^{p_0}} \bigg[\log\Big( 2+\frac{1}{|E_k|}\Big)\bigg]^{\frac{p_0}{p_1}(d'-1)}.
$$
Since $\int |E_k^{x''}|\ dx'' = |E_k|$,
we see that this is true for that part of the integral taken over those $x''$ which satisfy
$$
\Big( 1+{\sum}'' x_j \Big)^{-d'} |E_k^{x''}|^{-1} \le |E_k|^{-1}.
$$
For the remaining $x''$ we use the fact that the function 
$s [\log(2+s^{-1})]^{(d'-1)p_0\slash p_1}$ is essentially increasing in $s>0$. 
Taking $s = (1+{\sum}'' x_j)^{d'} |E_k^{x''}|< |E_k|$ we get an estimate by
$$
C \frac{|E_k|}{\lambda^{p_0}} \bigg[\log\Big( 2+\frac{1}{|E_k|}\Big)\bigg]^{\frac{p_0}{p_1}(d'-1)} 
	\int\Big(1+{\sum}'' x_j \Big)^{-d'}\, dx'',
$$
and the integral here is finite since $d'>d''$. The claim is justified, and it follows that
\begin{equation*}
\big|\big\{x : \psi_{d'}(x') U^{D'}_k \chi_E(x'') > \lambda\big\}\big| \lesssim
\frac{|E_k|}{\lambda^{p_0}} \bigg[\log\Big( 2+\frac{1}{|E_k|}\Big)\bigg]^{\frac{p_0}{p_1}(d'-1)}.
\end{equation*}

Now, in view of \eqref{dec}, with $\varepsilon > 0$ fixed, one gets
\begin{align*}
\big|\big\{x: \psi_{d'}(x') U^{D'}\chi_E(x) > \lambda\big\}\big| 
	& \le \sum_{k \ge 1} \big|\big\{x: 2^{-k}\psi_{d'}(x')U^{D'}_k \chi_E(x'')
		> c 2^{-\varepsilon k} \lambda\big\}\big| \\
& \lesssim \sum_{k \ge 1} 2^{-(1-\varepsilon)p_0 k} \frac{|E_k|}{\lambda^{p_0}} 
	\bigg[\log\Big( 2+\frac{1}{|E_k|}\Big)\bigg]^{\frac{p_0}{p_1}(d'-1)}.
\end{align*}
Using \eqref{iEk} and again the monotonicity of 
$s[\log(2+s^{-1})]^{(d'-1)p_0\slash p_1}$, 
we see that the last expression is controlled, in the sense of $\lesssim$, by
$$
\sum_{k \ge 1} 2^{-k[(1-\varepsilon)p_0-(1+\varepsilon)]} \frac{|E|}{\lambda^{p_0}} 
	\bigg[\log\bigg( 2+\frac{2^{-k(1+\varepsilon)}}{|E|}\bigg)\bigg]^{\frac{p_0}{p_1}(d'-1)}.
$$
Here we estimate the numerator in the argument of the logarithm by $1$ and sum the
geometric series. Choosing $\varepsilon$ small enough, we conclude that
$$
\big|\big\{x: S^{D'}\chi_E(x)> \lambda\big\}\big| \lesssim
\frac{|E|}{\lambda^{p_0}} 
	\bigg[\log\Big( 2+\frac{1}{|E|}\Big)\bigg]^{\frac{p_0}{p_1}(d'-1)}, \qquad \lambda > 0,
$$
which is precisely the desired estimate. The proof of Lemma \ref{dos} is complete.
\end{proof}

\begin{proof}[Proof of Lemma \ref{uno}]
The case $d'=1$ is essentially contained in the reasoning proving Lemma \ref{dos}.
Now, however, the situation is much simpler since no logarithmic factors appear.
\end{proof}

\subsection{The case of two minimal $\alpha_i$ in dimension $3$} \qquad \\
Without any loss of generality, we may assume that $\alpha=(a,a,b)$
with $-1<a<0$ and $a<b$.
\begin{thm} \label{pos_p0_s}
For $d=3$ and $\alpha$ as above, the operator $\M$ maps $L^{p_0,1}\log^{1\slash p_1}L$ into
$L^{p_0,\infty}$, in the sense that for all $E \subset \R^3_+$ of finite measure 
$$
\big|\{\M \chi_E > \lambda \}\big| \le C \frac{|E|}{\lambda^{p_0}}
								\bigg[ \log\bigg(2+ \frac{1}{|E|}\bigg) 
									\bigg]^{\frac{p_0}{p_1}}, \qquad \lambda > 0.
$$
\end{thm}

The proof of Theorem \ref{pos_p0_s} goes along similar lines to that of Theorem \ref{pos_s}.
When one estimates $\mathcal{H}_t^{(a,a,b)}(x,y)$, a sum of $8$ terms emerges;
as before, we index these terms by subsets $D' \subset \{1,2,3\}$ and let primed
variables correspond to the second term in \eqref{main}. Then only the term corresponding 
to $D'=\{1,2,3\}$ requires further analysis since, in view of \eqref{pqw},
the cases when $d'<3$ are covered by Lemmas \ref{dos} and \ref{uno}. 
Thus the task reduces to proving the following (we use $H_*$ and some other notations 
introduced in Section \ref{ssec:2m3}).
\begin{lem} \label{l_pos_p0_s}
For all $E \subset \R^3_+$ of finite measure, the distribution
function of $H_{*} \chi_E$ satisfies
$$
\big|\{H_* \chi_E > \lambda \}\big| \le C \frac{|E|}{\lambda^{p_0}}
								\bigg[ \log\bigg(2+ \frac{1}{|E|}\bigg) 
									\bigg]^{\frac{p_0}{p_1}}, \qquad \lambda > 0.
$$
\end{lem}

\begin{proof}
Recall from the proof of Lemma \ref{lpos_s} that
$$
H_* \chi_E(x) \lesssim \sum_{\nu,\beta \in \Z} \sum_{k \ge 0} H_{k,\nu,\beta} \chi_E(x),
	\qquad x \in \R^3_+.
$$
As a slight modification of \eqref{bs}, we now have
\begin{align*}
H_{k,\nu,\beta} \chi_E(x) \lesssim & \, 2^{(2a+3)k} 2^{-(\nu+\beta)b\slash 2}
	\exp\big(-c(2^{-\nu}+2^{-\beta})\big) (x_1 x_2 x_3)^{a\slash 2} 
	2^{(k+\nu)a\slash 2} \\ & \cdot \big[ 2^k (x_1+x_2+x_3)\big]^{-3(a+1)}
		\int_{y_3 \sim 2^{-k-\beta}} (y_1 y_2)^{a\slash 2} \exp\big(-c 2^k (y_1+y_2)\big)
			\chi_{E}(y)\, dy,
\end{align*}
where we used an exponential to get the factor preceding the integral. Therefore, denoting 
the last integral by $I(k,\beta,E)$ and invoking the function 
$\psi_d$ from Lemma \ref{psi}, we have
\begin{equation} \label{Hee}
H_{k,\nu,\beta} \chi_E(x) \lesssim 2^{-ka\slash 2} 
	2^{-(\nu+\beta)b \slash 2 +\nu a\slash 2} \exp\big(-c(2^{-\nu}+2^{-\beta})\big)
	\psi_3(x) I(k,\beta,E).
\end{equation}

Now we estimate $I(k,\beta,E)$. Given $y_3$, we introduce the slice 
$E^{y_3} = \{ (y_1,y_2) \in \R^2_+ : (y_1,y_2,y_3) \in E\}$. 
An application of Lemma \ref{drear} in dimension $d=2$ leads to
\begin{align*}
I(k,\beta,E) & = \int_{y_3 \sim 2^{-k-\beta}} dy_3 \int_{E^{y_3}}
	(y_1 y_2)^{a\slash 2} \exp\big(-c 2^k (y_1+y_2)\big) \, dy_1 dy_2 \\
& \lesssim \int_{y_3 \sim 2^{-k-\beta}} dy_3 \int_0^{|E^{y_3}|}
	\frac{1}{s^{1\slash p_1}} \bigg[\log\Big( 2+2^{-2k}\frac{1}{s}\Big)\bigg]^{1\slash p_1} \, ds \\
& \simeq \int_{y_3 \sim 2^{-k-\beta}} |E^{y_3}|^{1\slash p_0} 
	\bigg[\log\Big( 2+2^{-2k}\frac{1}{|E^{y_3}|}\Big)\bigg]^{1\slash p_1} \, dy_3.
\end{align*}
H\"older's inequality then implies
$$
I(k,\beta,E) \lesssim 2^{-(k+\beta)\slash p_1}
	\bigg( \int_{y_3 \sim 2^{-k-\beta}}
	|E^{y_3}| \bigg[\log\Big( 2+2^{-2k}\frac{1}{|E^{y_3}|}\Big)\bigg]^{p_0\slash p_1} dy_3 
		\bigg)^{1\slash p_0}.
$$

Combining the last estimate with \eqref{Hee} and simplifying the constant factors, we see that 
$$
H_{k,\nu,\beta} \chi_E(x) \lesssim 2^{-\varepsilon(|\nu|+|\beta|)}
	e^{-\delta 2^{-\beta}} \psi_3(x)
\bigg( \int_{y_3 \sim 2^{-k-\beta}}
	|E^{y_3}| \bigg[\log\Big( 2+2^{-2k}\frac{1}{|E^{y_3}|}\Big)\bigg]^{p_0\slash p_1} dy_3 
		\bigg)^{1\slash p_0},
$$
with $\varepsilon = (b-a)\slash 2 > 0$ and some $\delta>0$.
The support of $H_{k,\nu,\beta} \chi_E$ is contained in $\{x_3 \sim 2^{-k-\nu}\}$,
and this now allows us to estimate the level sets of $\sum_k H_{k,\nu,\beta} \chi_E$. Thus
\begin{align*}
& \Big|\Big\{x \in \R^3_+ : \sum_{k \ge 0} H_{k,\nu,\beta}\chi_E(x) > c 
	\lambda 2^{-\varepsilon (|\nu|+|\beta|)\slash 2} \Big\}\Big|  = 
	\sum_{k \ge 0} \big|\big\{H_{k,\nu,\beta}\chi_E(x) > c 
	\lambda 2^{-\varepsilon (|\nu|+|\beta|)\slash 2} \big\}\big| \le \\
& \sum_{k \ge 0} \bigg|\bigg\{ \psi_3(x) > c \lambda 
	2^{\varepsilon (|\nu|+|\beta|)\slash 2 + \delta 2^{-\beta}} 
	\bigg( \int_{y_3 \sim 2^{-k-\beta}} 
	|E^{y_3}| \bigg[\log\Big( 2+2^{-2k}\frac{1}{|E^{y_3}|}\Big)\bigg]^{p_0\slash p_1} dy_3 
		\bigg)^{-1\slash p_0} \bigg\}\bigg|.
\end{align*}
Since $\psi_3 \in L^{p_0,\infty}$ (Lemma \ref{psi}), this expression is controlled by
$$
2^{-\varepsilon (|\nu|+|\beta|)p_0\slash 2-\delta p_0 2^{-\beta}}
 \frac{1}{\lambda^{p_0}} \sum_{k\ge 0} \int_{y_3 \sim 2^{-k-\beta}}
	|E^{y_3}| \bigg[\log\Big( 2+2^{-2k}\frac{1}{|E^{y_3}|}\Big)\bigg]^{p_0\slash p_1} dy_3.	
$$

In order to estimate the above sum, we split the integral according to the condition
$2^{2k}|E^{y_3}|> |E|$ and then use the essential monotonicity of the function
$s [\log(2+s^{-1})]^{p_0\slash p_1}$, as before. This gives
\begin{align*}
& \int_{y_3 \sim 2^{-k-\beta}}
	|E^{y_3}| \bigg[\log\Big( 2+2^{-2k}\frac{1}{|E^{y_3}|}\Big)\bigg]^{p_0\slash p_1} dy_3 \\ 
& \lesssim\int_{y_3 \sim 2^{-k-\beta}} |E^{y_3}| \, dy_3 \;
	\bigg[\log\Big( 2+\frac{1}{|E|}\Big)\bigg]^{p_0\slash p_1} +
	\int_{y_3 \sim 2^{-k-\beta}} \, dy_3 \; 2^{-2k} |E|
	\bigg[\log\Big( 2+\frac{1}{|E|}\Big)\bigg]^{p_0\slash p_1}. 
\end{align*}
Summing the right-hand side here in $k \ge 0$ gives at most
$$
C \big(1 + 2^{-\beta}\big) |E| \bigg[\log\Big( 2+\frac{1}{|E|}\Big)\bigg]^{p_0\slash p_1}.
$$

Combining this with the previous considerations, we arrive at
$$
\Big|\Big\{x \in \R^3_+ : \sum_{k \ge 0} H_{k,\nu,\beta}\chi_E(x) > c 
	\lambda 2^{-\varepsilon (|\nu|+|\beta|)\slash 2} \Big\}\Big| \lesssim
	2^{-\varepsilon(|\nu|+|\beta|)p_0 \slash 2} 
	\frac{|E|}{\lambda^{p_0}} \bigg[\log\Big( 2+\frac{1}{|E|}\Big)\bigg]^{p_0\slash p_1}.
$$
Finally, the exponentially decreasing factor allows us to sum in $\nu$ and $\beta$:
\begin{align*}
|\{x \in \R^3_+ : H_{*}\chi_E(x) > \lambda\}| & \le \sum_{\nu,\beta \in \Z}
\Big|\Big\{x \in \R^3_+ : \sum_{k \ge 0} H_{k,\nu,\beta}\chi_E(x) > c 
	\lambda 2^{-\varepsilon (|\nu|+|\beta|)\slash 2} \Big\}\Big| \\
& \lesssim \frac{|E|}{\lambda^{p_0}} \bigg[\log\Big( 2+\frac{1}{|E|}\Big)\bigg]^{p_0\slash p_1}.
\end{align*}
The desired estimate follows, and the proof of Lemma \ref{l_pos_p0_s} is complete.
\end{proof}

\subsection{Counterexamples} \qquad \\
Assume now that $d\ge 4$ and that there are at least two minimal values $\alpha_i$. 
We shall construct sets $E$ to prove the negative part of Theorem \ref{main_thm} (b3). 
Here we may replace $T_*^{\alpha}$ by $\M$, since only $t \le 1$ will be taken into account.
For the sake of clarity, we state the result separately.

\begin{thm} \label{neg_p0}
Let $d \ge 4$ and $\alpha$ as above. There are no $C>0$ and $\gamma \in \R$ such that
$$
\big|\{\M \chi_E > \lambda \}\big| \le C \frac{|E|}{\lambda^{p_0}}
								\bigg[ \log\bigg(2+ \frac{1}{|E|}\bigg) 
									\bigg]^{\gamma}, \qquad \lambda > 0,
$$
holds for all $E \subset \R^d_+$ of finite measure.
\end{thm}

We will prove the theorem in the case when all $\alpha_i$ are minimal, $\alpha_i = a \in (-1,0)$.
The same reasoning works in the general case, if we include the variables corresponding
to non-minimal $\alpha_i$ among the double-primed variables below. 

\begin{proof}[Proof of Theorem \ref{neg_p0}]
Assume first that $d \ge 5$ and choose $d'$ so that $2\le d' < d''$.
We shall find a family of sets $E$ disproving the estimate stated in Theorem \ref{neg_p0}. 

Given a parameter $\beta >0$ (specified in a moment), consider the set
\begin{equation} \label{fEt}
E_t = \Big\{ y \in \R^d_+ : {\prod}' y_j^{-1\slash p_1} > \beta, \;\; y_j < t \;\; 
\textrm{for} \;\;
	j \le d', \;\; t^{-1} < y_j < 2t^{-1} \;\; \textrm{for} \;\; j>d' \Big\}.
\end{equation}
By Lemma \ref{level} (a) and (b), applied in $\R^{d'}_+$, we obtain
\begin{equation} \label{Emes}
|E_t| \simeq \beta^{-p_1} \big[\log(2+t^{d'}\beta^{p_1})\big]^{d'-1} t^{-d''}, \qquad t^{d'} \beta^{p_1}>1.
\end{equation}
Next we estimate $\M \chi_{E_t}(x)$. We consider $x$ such that $x_j < t$ for $j \le d'$
and $t^{-1} < x_j < 2t^{-1}$ for $j>d'$. Further, below we restrict the integration to the 
set of points $y \in E_t$ with $|y_j-x_j|<1$ for $j>d'$, so that in view of items
(a) and (b) of Lemma \ref{lower}, for $t \le 1\slash 4$, 
$$
{\prod}' \mathcal{H}_t^a(x_i,y_i) \gtrsim \Big({\prod}' x_i^{-1\slash p_1}\Big)
	\Big({\prod}' y_i^{-1\slash p_1}\Big) t^{(2\slash p_1 -1)d'} \quad \textrm{and} \quad 
{\prod}'' \mathcal{H}_t^a(x_i,y_i) \gtrsim 1.
$$
Thus
\begin{align} \nonumber
\M \chi_{E_t}(x) & \gtrsim \Big({\prod}' x_i^{-1\slash p_1}\Big) t^{(2\slash p_1 -1)d'}
	\int_{E_t \cap \{|y_j-x_j|<1 \; \textrm{for}\; j>d'\}} 
	{\prod}' y_i^{-1\slash p_1} \, dy \\
& \gtrsim t^{-d'\slash p_1} t^{(2\slash p_1 -1)d'} \beta
	\big| E_t \cap \{ |y_j-x_j| < 1 \; \textrm{for} \; j > d'\}\big|
	\nonumber \\
& \simeq t^{d''-d'\slash p_0} \beta |E_t|. \label{crss}
\end{align}
Consequently, taking $\lambda \simeq t^{d''-d'\slash p_0} \beta |E_t|$ below, we get
$$
\lambda |\{\M \chi_{E_t}> \lambda \}|^{1\slash p_0} \gtrsim
	t^{d''-d'\slash p_0} \beta |E_t| t^{d'\slash p_0} t^{-d''\slash p_0} 
	= t^{d''\slash p_1} \beta |E_t|.
$$
Then, with the choice $\beta = t^{-d''\slash p_1}$, we have 
$$
\lambda^{p_0} |\{\M \chi_{E_t}> \lambda \}|
\bigg(|E_t| \bigg[\log\Big(2+\frac{1}{|E_t|}\Big)\bigg]^{\gamma} \,\bigg)^{-1}
\gtrsim |E_t|^{p_0\slash p_1} \bigg[\log\Big(2+\frac{1}{|E_t|}\Big)\bigg]^{-\gamma}.
$$
But the last expression tends to $\infty$ as $t\to 0^+$ since, by \eqref{Emes}, $|E_t| \to \infty$
as $t \to 0^+$. This finishes the case $d''>d'\ge 2$.

When $d'=d''\ge 2$ and in particular when $d=4$, the construction is a little more
complicated, as in Section \ref{cex_p1}. Let $E_{2^{-j}}$ be the set given by \eqref{fEt} 
with $t=2^{-j}$ and $\beta = N^{1\slash p_1} t^{-d'\slash p_1}$. Then define
$$
F_N = \bigcup_{j=2}^N E_{2^{-j}},
$$
which is a disjoint sum. From \eqref{Emes} it follows that
$$
|F_N| = \sum_{j=2}^{N} |E_{2^{-j}}| \simeq \sum_{j=2}^N N^{-1} \big[\log(2+N)\big]^{d'-1} 
	\simeq \big[\log(2+N)\big]^{d'-1}.
$$
Next, we use \eqref{crss} to estimate $\M \chi_{F_N}(x)$ when
$x\in (2^{-j-1},2^{-j})^{d'}\times (2^j,2^{j+1})^{d'}$ and $j=2,\ldots,N$ 
(notice that each of these sets has measure essentially $1$), obtaining
$$
\M \chi_{F_N}(x) \gtrsim \M \chi_{E_{2^{-j}}}(x)  \gtrsim (2^{-j})^{d'-d'\slash p_0} 
	N^{1\slash p_1} (2^{-j})^{-d'\slash p_1} |E_{2^{-j}}| \simeq N^{-1\slash p_0} [\log(2+N)]^{d'-1}.
$$
Thus, choosing $\lambda \simeq N^{-1\slash p_0}[\log(2+N)]^{d'-1}$ below,
$$
\lambda^{p_0} |\{\M \chi_{F_N}>\lambda\}| \gtrsim \big[\log(2+N)\big]^{(d'-1)p_0}.
$$
Since we also have
$$
|F_N| \bigg[\log\Big(2+\frac{1}{|F_N|}\Big)\bigg]^{\gamma} \lesssim \big[\log(2+N)\big]^{d'-1},
$$
we conclude that
\begin{align*}
\lambda^{p_0} |\{\M \chi_{F_N}>\lambda\}|\bigg(|F_N|
 \bigg[\log\Big(2+\frac{1}{|F_N|}\Big)\bigg]^{\gamma}\, \bigg)^{-1} 
	& \gtrsim \frac{[\log(2+N)]^{(d'-1)p_0}}{[\log(2+N)]^{d'-1}} \\
	& = \big[\log(2+N)\big]^{(d'-1)p_0\slash p_1},
\end{align*}
and the last expression tends to $\infty$ as $N\to \infty$. The proof is finished.
\end{proof}

\subsection{Comment on sharpness} \label{sharp_P0}
Recall that the space $L^{p_0,1}\log^{(\widetilde{d}(\alpha)-1)\slash p_1} L$ 
is defined by the inequality
$$
\int_0^{\infty} f^*(s) s^{1\slash p_0} \bigg[\log \Big(2+\frac{1}{s}\Big)
 \bigg]^{{(\widetilde{d}(\alpha)-1)}\slash {p_1}}\, \frac{ds}{s} < \infty.
$$
In Theorem \ref{main_thm} (a3) and (b3), this space is best possible in the sense
of convergence at $0$ of the above integral. Indeed, if a space $X$ of measurable functions 
in $\R^d_+$ is invariant under rearrangement and contains a function $g$ with 
\begin{equation} \label{nsn}
\int_0^{1} g^*(s) s^{1\slash p_0} \bigg[\log\Big(2+\frac{1}{s}\Big)
	\bigg]^{(\widetilde{d}(\alpha)-1) \slash p_1}\, \frac{ds}{s} = \infty,
\end{equation}
then $X$ also contains a function $f$ with
$$
T^{\alpha}_* f = T^{\alpha}_1 f = 
	\int \mathcal{H}^{\alpha}_{1\slash 2}(\cdot,y) f(y)\, dy = \infty 
$$
on a set of positive measure. We shall verify this fact in the case when all $\alpha_i$ 
are minimal, so that $\widetilde{d}(\alpha)=d$. 
The general case requires only minor modifications; see Section \ref{sharp_p1}.

First, from Lemma \ref{lower} (b) we see that it is enough to find an $f \in X$ such that 
$$
\int_{(0,1)^d} f(y) \prod_{j=1}^d y_j^{-1\slash p_1} \, dy = \infty.
$$
Let $\Psi(y) = \prod_1^d y_j^{-1\slash p_1}$ for $y \in (0,1)^d$.
Lemma \ref{level} (b) says that the level sets 
$\{y \in (0,1)^d : \Psi(y)> \lambda\}$ of $\Psi$ have measures at least
$c \lambda^{-p_1} [\log(2+\lambda^{p_1})]^{d-1}$ for large $\lambda$.
From this it is elementary to see that the decreasing rearrangement $\Psi^*$ satisfies
$$
\Psi^{*}(s) \ge c s^{-1\slash p_1} \big[\log(2+s^{-1})\big]^{(d-1)\slash p_1}
$$
for small $s>0$. Our assumption \eqref{nsn} thus means that $\int g^*(s) \Psi^*(s)\, ds = \infty$,
and $g \in X$. But for a suitable rearrangement $f$ of $g$, in the sense that $f^{*} = g^{*}$, 
one can achieve equality in the classical inequality
$$
\int_{(0,1)^d} f(y) \Psi(y) \, dy \le \int_0^1 f^{*}(s)\Psi^{*}(s) \,ds,
$$
already mentioned in the proof of Proposition \ref{estimate}.
This implies that the left-hand integral here diverges, as desired.

The above sharpness remark shows, in particular, that $T^{\alpha}_*$
is not of weak type $(p_0,p_0)$, even if there is only one minimal $\alpha_i$.


\end{document}